\documentclass[12pt,fleqn]{article}
\usepackage{mathtools,amsfonts,amsmath, amsthm,amssymb}

\usepackage{graphicx}
\usepackage{caption}
\usepackage{subcaption}
\usepackage[title,titletoc,toc]{appendix}
\newenvironment{dedication}
        {\vspace*{0.1ex}
        \begin{quotation}\begin{center}\begin{em}}
        {\par\end{em}\end{center}\end{quotation}}

\def\lab{\label}
\def\n{\bf N}
\def\s{\bf S}
\usepackage[active]{srcltx}

\usepackage{fouriernc}

\usepackage{latexsym}
\usepackage{esint}
\usepackage[margin=22mm]{geometry} 
 \numberwithin{equation}{section}
\usepackage{color}
 \definecolor{db}{rgb}{0.0,0.0,0.8} 
\definecolor{dg}{rgb}{0.0,0.55,0.14}
\definecolor{dr}{rgb}{0.5,0,0.07}
  \usepackage{hyperref}
 \hypersetup{frenchlinks=true,colorlinks=true,linkcolor=db,citecolor=dg,
 bookmarks=true}
 \usepackage{enumerate}
\usepackage{authblk}

\newtheorem{theorem}{Theorem}
\newtheorem{proposition}{Proposition}[section]

\newtheorem{lemma}[proposition]{Lemma}
\newtheorem{corollary}[proposition]{Corollary}
\theoremstyle{definition}

\theoremstyle{definition}

\theoremstyle{definition}

\theoremstyle{definition}

\theoremstyle{definition}

\theoremstyle{definition}
\newtheorem{remark}[proposition]{Remark}
\newtheorem{theo}[proposition]{Theorem}
\theoremstyle{definition}

\newtheorem{open-problem}{Open Problem}
\newtheorem*{open-problem*}{Open Problem 2'}
\newcounter{step}

\newcommand{\rlemma}[1]{Lemma~\ref{#1}}
\newcommand{\rth}[1]{Theorem~\ref{#1}}

\newcommand{\rprop}[1]{Proposition~\ref{#1}}

 \def\wsp{W^{s,p}}
\def\be{\begin{equation}}
\def\ee{\end{equation}}
\def\bes{\begin{equation*}}
\def\ees{\end{equation*}}
\def\bt{\begin{theorem}}
\def\et{\end{theorem}}
\def\bpr{\begin{proposition}}
\def\epr{\end{proposition}}
\def\bl{\begin{lemma}}
\def\el{\end{lemma}}
\def\bc{\begin{corollary}}
\def\ec{\end{corollary}}
\def\br{\begin{remark}}
\def\er{\end{remark}}
\def\ben{\begin{enumerate}}
\def\bena{\begin{enumerate}[a)]}
\def\een{\end{enumerate}}
\def\bit{\begin{itemize}}
\def\iit{\end{itemize}}
\def\dist{\operatorname{dist}}
\def\det{\operatorname{det}}
\def\Jac{\operatorname{Jac}}

\def\deg{\operatorname{deg}}

\def\Dist{\operatorname{Dist}}
\def\sgn{\operatorname{sgn}}

\DeclareMathAlphabet{\mathonebb}{U}{bbold}{m}{n}
\newcommand{\one}{\ensuremath{\mathonebb{1}}}
\def\R{{\mathbb R}}

\def\C{{\mathbb C}}
\def\Z{{\mathbb Z}}

\usepackage{csquotes}
\def\fo{\forall\, }

\def\va{\varphi}

\def\d{\displaystyle}
\def\im{\imath}
\def\ve{\varepsilon}
\def\p{\partial}

\def\na{\nabla}

\def\so{{\mathbb S}^1}
\def\st{{\mathbb S}^2}
\def\sn{{\mathbb S}^N}

\def\bfs{{\bf s}}
\def\bft{{\bf t}}
\newtheorem{op}{Open Problem}

\newtheorem{appl}{Lemma}

\date{\today}
\title{Distances between homotopy classes of $W^{s,p}(\sn ;{\mathbb S}^N)$}
\author[1,3]{Haim Brezis}
\author[2]{Petru Mironescu}
\author[3]{Itai Shafrir}

\affil[1]{Department of Mathematics, Rutgers
    University,  USA}
\affil[2]{Universit\'e de Lyon,  Universit\'e Lyon 1, CNRS UMR 5208,  Institut Camille Jordan,  69622 Villeurbanne, France}
\affil[3]{Department of Mathematics, Technion - I.I.T., 32 000 Haifa, Israel}

\begin{document}
%%%%%%title page continued
\maketitle

\begin{dedication}
\hspace{33mm}
%\vspace*{2cm}
{Dedicated to Jean-Michel Coron with esteem and affection} 
\end{dedication}
%\begin{abstract}
%X. 
%\footnotetext{\emph{Key words} X}     
%\footnotetext{\emph{Subject classification} X}
%\end{abstract}
%\renewcommand{\thefootnote}{\arabic{footnote}} 
%%%%%%title page continued
\section{Introduction}
\lab{s1}

In \cite{bc}, J.-M. Coron and the first author (H. B.) have investigated the existence of multiple ${\mathbb S}^2$-valued harmonic maps. In the process they were led to introduce a concept of topological degree for maps $f\in H^1 ({\mathbb S}^2 ; {\mathbb S}^2)$. Note that such maps need not be continuous and thus the standard degree (defined for continuous maps) is not well-defined. Instead they used Kronecker's formula
\be
\label{wa1}
\deg f=\fint_{{\mathbb S}^2}\det\, (\na f)
\ee
 valid for $f\in C^1({\mathbb S}^2 ; {\mathbb S}^2)$, and a density argument ($C^1({\mathbb S}^2 ; {\mathbb S}^2)$ is dense in $H^1({\mathbb S}^2 ; {\mathbb S}^2)$) due to R. Schoen and K. Uhlenbeck \cite{su}, to assert that $\deg f$, defined by \eqref{wa1}, belongs to $\Z$ for every $f\in H^1({\mathbb S}^2 ; {\mathbb S}^2)$. 

They also used the technique of \enquote{bubble insertion} which allows to modify the degree $d_1$ of a given (smooth) map $f:{\mathbb S}^2\to{\mathbb S}^2$ by changing its values in a small disc $B_\ve(x_0)$. More precisely (see \cite{bc} and \cite{b-L}), for any $\ve>0$ and $d_2\in\Z$ one can construct some $g\in H^1({\mathbb S}^2 ; {\mathbb S}^2)$ such that $g=f$ outside $B_\ve(x_0)$, $\deg g=d_2$, and 
\be
\label{wa2}
\int_{{\mathbb S}^2}|\na g-\na f|^2\leq 8\pi\, |d_2-d_1|+o(1)\text{ as }\ve\to 0
\ee
(in fact \cite{bc} contains a more refined estimate in the spirit of Lemma \ref{ic} below). This kind of argument serves as a major source of inspiration for several proofs in this paper. As we are going to see, estimate \eqref{wa2} provides a useful upper bound for the Hausdorff distance between homotopy classes in $H^1(\st ; \st)$. 

Subsequently the first author and L. Nirenberg \cite{brezisnirenberg1} (following a suggestion of L. Boutet de Monvel and O. Gabber \cite[Appendix]{bgp}) developed a concept of topological degree for map in VMO$\, (\sn ; \sn)$ which applies in particular to the (integer or fractional) Sobolev spaces $\wsp (\sn ; \sn)$ with
\be
\label{ua1}
s>0,\ 1\leq p<\infty\text{ and }sp\geq N.
\ee

This degree is stable with respect to strong convergence in BMO and coincides with the usual degree when maps are smooth. 

In the remaining cases, i.e., when $sp<N$, there is no natural notion of degree. Indeed, one may construct a sequence of smooth maps $f_n:\sn\to\sn$  such that $f_n\to P$ (with $P\in\sn$ a fixed point) in $\wsp$ and $\deg f_n\to\infty$ \cite[Lemma 1.1]{bbmweb}. Therefore, in what follows we make the assumption \eqref{ua1}.

Given any $d\in\Z$, consider the classes
\be
\label{bgm8}
{\cal E}_{d}:=\{ f\in W^{s,p}(\sn ; \sn);\, \deg f=d\};
\ee
these classes depend not only on $d$, but also on $s$ and $p$, but in order to keep notation simple we do not mention the dependence on $s$ and $p$.

These classes are precisely the connected or path-connected components of $\wsp (\sn ; \sn)$. [This was proved in \cite{brezisnirenberg1} in the VMO context, but the proof can be adapted to $\wsp$.] Moreover if $N=1$ we have (see Section \ref{sec:ua2}) 
\be
\label{ua3}
{\cal E}_{d}=\left\{ f;\, f(z)=e^{\im\va(z)}\, z^d,\text{ with } \va\in W^{s,p}(\so ; \R)\right\}.
\ee

Our purpose is to investigate the usual distance and the Hausdorff
distance (in $\wsp$) between the classes ${\cal E}_d$. For that matter
we introduce the $W^{s,p}$-distance between two maps $f,g\in \wsp (\sn
; \sn)$ by 
\be
\label{va1}
d_{\wsp}(f, g):=|f-g|_{\wsp},
\ee 
where for  $h\in \wsp(\sn ; \R^{N+1})$ we let 
\bes
|h|_{\wsp}:=\left\| h-\fint_{\sn} h \right\|_{\wsp},
\ees
and  $\|\ \|_{\wsp}$ is any one of the standard norms on
$\wsp$. Let $d_1\neq d_2$ and  define the following
two quantities:
\be
\label{bgm5}
\dist_{\wsp} ({\cal E}_{d_1}, {\cal E}_{d_2}):=\inf_{f\in{\cal E}_{d_1}}\ \inf_{g\in{\cal E}_{d_2}}\ d_{\wsp}(f, g)\,,
\ee
and
\be
\label{sa1}
\Dist_{\wsp}({\cal E}_{d_1}, {\cal E}_{d_2}):=\sup_{f\in{\cal E}_{d_1}}\ \inf_{g\in{\cal E}_{d_2}}\ d_{\wsp}(f, g)\,.
\ee
It is conceivable that 
\be
\label{eq:sym}
\Dist_{\wsp}({\cal E}_{d_1}, {\cal E}_{d_2})=\Dist_{\wsp}({\cal E}_{d_2}, {\cal E}_{d_1}), \fo d_1, d_2\in\Z,
\ee
but we have not been able to prove this equality (see Open Problem
\ref{sa2} below).
Therefore we consider also the symmetric version of
\eqref{sa1}, which is nothing but the Hausdorff distance between the
two classes: 
\be
\label{bgm6}
H-\dist_{\wsp} ({\cal E}_{d_1}, {\cal E}_{d_2})=
\max\left\{  \Dist_{\wsp}({\cal E}_{d_1}, {\cal E}_{d_2}), \Dist_{\wsp}({\cal E}_{d_2}, {\cal E}_{d_1})  \right\}.
\ee
We should mention that even in cases where we know that
\eqref{eq:sym} holds true, the qualitative properties of the
 two quantities in \eqref{eq:sym} might be
quite different. Consider for example the classes ${\cal E}_{d_1},
{\cal E}_{d_2}$ in $W^{1,1}(\so;\so)$ when $0<d_1< d_2$. It is shown
in \rprop{prop:attain}
 that $\Dist_{W^{1,1}}({\cal E}_{d_1}, {\cal
  E}_{d_2})$ is attained by some $f$ and $g$, while $\Dist_{W^{1,1}}({\cal E}_{d_2}, {\cal
  E}_{d_1})$ is not.

It turns out that  in general the analysis of the usual distance
$\dist_{\wsp}$ is simpler than that of $\Dist_{\wsp}$, so we start with it.
Note  that we  clearly have
\be
\label{vl1}
\dist_{C^0}({\cal E}_{d_1}, {\cal E}_{d_2})=2,\ \fo d_1\neq d_2.
\ee

Indeed, on the one hand we have $\|f-g\|_{C^0}\leq 2$, $\fo f, g$, and on the other hand if $\|f-g\|_{C^0}<2$ then $\deg f=\deg g$. [This is obtained by considering the homotopy $\d H_t=\frac{tf + (1-t)g}{|tf + (1-t)g|}$, $t\in [0,1]$.] By contrast, it was established in \cite{brezisnirenberg1} that surprisingly, when $s=1/2$, $p=2$ and $N=1$ one has $\dist_{H^{1/2}}({\cal E}_1, {\cal E}_0)=0$, and thus $\dist_{\textrm{VMO}}({\cal E}_1, {\cal E}_0)=0$. 
The usual distance $\dist_{\wsp}({\cal E}_{d_1},{\cal E}_{d_2})$ in certain (non-fractional) Sobolev
spaces was investigated  in works by J. Rubinstein and I. Shafrir
\cite{rs}, when $s=1$, $p\geq N=1$, and S. Levi and I. Shafrir
\cite{ls}, when $s=1$, $p\geq N\geq 2$. In particular, they obtained  exact formulas for the distance (see \cite[Remark 2.1]{rs}, \cite[Theorem 3.4]{ls}) and tackled the question whether this distance is achieved (see \cite[Theorem 1]{rs}, \cite[Theorem 3.4]{ls}). Another motivation comes from the forthcoming paper \cite{bms}, where we consider a natural notion of class in $W^{1,1}(\Omega ; \so)$ (with $\Omega\subset\R^N$) and determine the distance between these classes. In particular, Theorem \ref{vo1} is used in \cite{bms}.

Throughout most of the paper we assume that $N=1$. It is only in the
last  two sections that we consider $N\geq 2$.

We pay  special attention to the case where  $s=1$. In this case, we have several sharp results when
%, which are sensitive to the choice of the specific $W^{1,p}$ semi-norm.  In the $W^{1,p}$ frame, we set
we take
\be
\label{ub2}
d_{W^{1,p}}(f, g)=|f-g|_{W^{1,p}}:=\left(\int_{\so}|\dot f-\dot g|^p\right)^{1/p}.
\ee

%We start by investigating the $\wsp$ distance $\dist_{\wsp}$. 
%
%In $W^{1,p}(\so ; \so)$, t
The following result was obtained in \cite{rs}.
\bt
\label{bgl20}
Let $1\leq p<\infty$. We have
\be
\label{bgl30}
\dist_{W^{1,p}} ({\cal E}_{d_1}, {\cal E}_{d_2})=2^{(1/p)+1}\pi^{(1/p)-1}\, |d_1-d_2|,\ \fo d_1, d_2\in\Z.
\ee

In particular
\be
\label{1a1}
\dist_{W^{1,1}}({\cal E}_{d_1}, {\cal E}_{d_2})=4\, |d_1-d_2|,\ \fo d_1, d_2\in\Z.
\ee
\et

%\bpr
%\label{bgl2}
% We have
%\be
%\label{bgl3}
%\dist_{W^{1,1}} ({\cal E}_{d_1}, {\cal E}_{d_2})=4\, |d_1-d_2|,\ \fo d_1, d_2\in\Z.
%\ee
%\epr

For the convenience of the reader, and also because it is used in the proof of Theorem \ref{ub90}, the proof of Theorem \ref{bgl20} is presented in Sections \ref{sec:ua5} and \ref{ub6}.

In view of \eqref{bgl30}, it is natural to ask whether, given $d_1\neq d_2$, the infimum  
\be
\label{ub80}
\inf_{f\in{\cal E}_{d_1}}\ \inf_{g\in{\cal E}_{d_2}}\ d_{W^{1,p}}(f, g)=2^{(1/p)+1}\pi^{(1/p)-1}\, |d_1-d_2|
\ee
is achieved. 
The answer is given by the following result, proved in \cite{rs} when $p=2$.
\bt
\label{ub90}
Let $d_1, d_2\in\Z$, $d_1\neq d_2$.
\ben
\item
When $p=1$, the infimum in \eqref{ub80} is always achieved.
\item
When $1<p<2$, the infimum in \eqref{ub80} is achieved if and only if $d_2=-d_1$.
\item
When $p\geq 2$, the infimum in \eqref{ub80} is not achieved.
\een
\et

%\bt
%\label{ub9}
%The infimum in \eqref{ub8} is  achieved for every $d_1, d_2\in\Z$.
%\et

We now turn to the case $s\neq 1$.  Here, we will only obtain the
order of magnitude of the distances $\dist_{\wsp}$, and thus our
results are not sensitive to the choice of a specific distance among
various equivalent ones. [However, we will occasionally obtain sharp results for $H^{1/2}(\so ; \so)$ equipped with the Gagliardo distance defined below.] When $0<s<1$ a standard distance is
associated with the Gagliardo $\wsp$ semi-norm
\be
\label{uc8}
d_{W^{s,p}}(f,g):=\left(\int_{\so}\int_{\so}\frac{|[f(x)-g(x)]-[f(y)-g(y)]|^p}{|x-y|^{1+sp}}\, dxdy  \right)^{1/p}\,.
\ee
%Given a standard $\wsp$ 
%%norm or
% semi-norm, we let  
%% $d_{\wsp}(f,g)=\|f-g\|_{\wsp}$ (or 
% $d_{\wsp}(f,g)=|f-g|_{\wsp}$ and define $\dist_{\wsp}$ and $H-\dist_{\wsp}$ as in \eqref{bgm5}--\eqref{bgm6}. 

%Then the following results cover all the remaining cases, i.e., $s\neq 1$ and $s\neq 1/p$.
%1/p$, $1<p<\infty$.
%Another case where we can explicitly calculate the function $\dist$ is $1<p<\infty$ and $s=1/p$. 

\bt
\label{uc3} 
We have
\ben
\item
Let $1<p<\infty$. Then 
\be
\label{uc6}
\dist_{W^{1/p, p}} ({\cal E}_{d_1}, {\cal E}_{d_2})=0,\ \fo d_1, d_2\in\Z.
\ee
\item
Let $s>0$ and $1\leq p<\infty$ be such that $sp>1$. Then
\be
\label{ue2}
C'_{s, p}\, |d_1-d_2|^{s}\leq \dist_{\wsp} ({\cal E}_{d_1}, {\cal E}_{d_2})\leq C_{s, p}\,  |d_1-d_2|^s
\ee
for some constants $C_{s,p}, C'_{s, p}>0$.
\een
\et
%For $\dist_{W^{1/p, p}}$ with $p\geq 2$, the above result appears in  \cite{brezisnirenberg1}.

{}
\medskip

We next investigate the Hausdorff distance $H-\dist_{\wsp}$ (still with $N=1$). 
\bt
\label{vm1}
We have
\ben
\item In $W^{1,1}$,
\be
\label{XX}
\Dist_{W^{1,1}} ({\cal E}_{d_1}, {\cal E}_{d_2})= 2\pi|d_1-d_2|,\  \fo d_1, d_2\in\Z.
\ee
%\item
%If $d_1 d_2\geq 0$,  then $H-\dist_{W^{1,1}} ({\cal E}_{d_1}, {\cal E}_{d_2})= 2\pi|d_1-d_2|$.
\item
If $1<p<\infty$, then 
\be
\label{uc60}
H-\dist_{W^{1/p, p}} ({\cal E}_{d_1}, {\cal E}_{d_2})\leq C_p\, |d_1-d_2|^{1/p},\ \fo d_1, d_2\in\Z.
\ee
\item
If $s>0$ and $1\leq p<\infty$ are such that $sp>1$,  then
\be
\label{vn1}
\Dist_{W^{s,p}} ({\cal E}_{d_1}, {\cal E}_{d_2})=\infty,\ \fo d_1, d_2\in\Z\text{ such that }d_1\neq d_2.
\ee
\een
\et

%\bpr
%The following hold.
%\label{bgm2}
%\ben
%\item
%$H-\dist_{W^{1,1}} ({\cal E}_{d_1}, {\cal E}_{d_2})\leq 2\pi|d_1-d_2|$, $ \fo d_1, d_2\in\Z$.
%\item
%When $d_1 d_2\geq 0$, we have $H-\dist_{W^{1,1}} ({\cal E}_{d_1}, {\cal E}_{d_2})= 2\pi|d_1-d_2|$.
%\een
%\epr

%We do not know whether the results in items {\it 1} and {\it 3} above are optimal: 
% \begin{open-problem}
%\label{ib}
% Is it true that, when $d_1 d_2 <0$,  we  have  $H-\dist_{W^{1,1}}({\cal E}_{d_1},{\cal E}_{d_2})=2\pi|d_1-d_2|$?
% 
%  Recall that, by \eqref{1a1} and Theorem \ref{vm1}, item {\it 1}, we have 
% \bes
% 4 |d_1-d_2|\le H-\dist_{W^{1,1}}({\cal E}_{d_1},{\cal E}_{d_2})\le 2\pi|d_1-d_2|,\ \fo d_1, d_2\in\Z.
% \ees
%\end{open-problem}
%\begin{open-problem}
%\label{vd3}
%Is it true that for every $1<p<\infty$  we have
%\be
%\label{uf3}
% H-\dist_{W^{1/p, p}} ({\cal E}_{d_1}, {\cal E}_{d_2})\geq C'_{p}\, |d_1-d_2|^{1/p},\ \fo d_1, d_2\in\Z,
%\ee
%for some positive constant $C'_p$?
% \end{open-problem}
%
%
%
% A very partial answer to the second open problem is the following. 
%For every $d_2\in\Z$, there exists some $C'_{p, d_2}>0$ such that  
%\be
%\label{uf1}
%H-\dist_{W^{1/p, p}} ({\cal E}_{d_1}, {\cal E}_{d_2})\geq C'_{p, d_2}\, |d_1-d_2|^{1/p},\ \fo d_1\in\Z.
%\ee
%
%The proof of \eqref{uf1} is given in Section \ref{ua6}.
%
%
%Concerning the Hausdorff distance, we will discuss, in Section \ref{vc6}, the  more subtle question whether the  $H-\dist_{W^{1,1}}$ is \enquote{achieved}.
%

We do not know whether \eqref{uc60} is optimal in the sense that for every $1<p<\infty$  we have
\be
\label{uf3}
 \Dist_{W^{1/p, p}} ({\cal E}_{d_1}, {\cal E}_{d_2})\geq C'_{p}\, |d_1-d_2|^{1/p},\ \fo d_1, d_2\in\Z,
\ee
for some positive constant $C'_p$. See Open Problem \ref{sa3} below for a more general question. See also Section \ref{sa4} for some partial positive answers.

\medskip
We now discuss similar questions when $N\geq 2$. We define $\dist_{\wsp}$ and $H-\dist_{\wsp}$ using one of the usual $\wsp$ (semi-)norms. 

 For $s=1$, $N\geq 2$, $p\geq N$, and for the semi-norm $|f-g|_{W^{1,p}}=\|\na f-\na g\|_{L^p}$, the exact value of the  $W^{1,p}$ distance $\dist_{W^{1,p}}$ between the classes ${\cal E}_{d_1}$ and ${\cal E}_{d_2}$, $d_1\neq d_2$,  has been computed by S. Levi and I. Shafrir \cite{ls}. A striking fact is that this distance does not depend on $d_1$ and $d_2$, but only on $p$ (and $N$). 
 
% Our results in this context are the following extensions of the results in \cite{ls}.

We start with $\dist_{\wsp}$.  
%$W^{N,1}$, which plays the role of $W^{1,1}$.
\bt
\label{vo1}
We have
\ben
\item
If $N\geq 1$ and $1< p<\infty$, then 
\be
\label{vo2}
\dist_{W^{N/p,p}}({\cal E}_{d_1}, {\cal E}_{d_2})=0,\ \fo d_1, d_2\in\Z.
\ee
\item
If [$1<p<\infty$ and $s>N/p$] or [$p=1$ and $s\ge N$], there exist constants $C_{s,p, N},  C'_{s,p, N}>0$ such that
\be
\label{ve5}
C'_{s,p, N}\leq \dist_{W^{s,p}}({\cal E}_{d_1}, {\cal E}_{d_2})\leq
C_{s,p, N},\ \fo d_1, d_2\in\Z\text{ such that }d_1\neq d_2,
\ee
  (here $N\geq 2$ is essential).
\een 
\et
% %\todo{Itai: Can you provide the proof of Theorem \ref{vo1}?}
\br
\label{vo6}
We do not know whether, under the  assumptions of Theorem \ref{vo1},
item {\it 2},  it is true that
$\dist_{\wsp}({\cal E}_{d_1}, {\cal E}_{d_2})=C''_{s, p, N}$, $\fo d_1, d_2\in\Z$ such that $d_1\neq d_2$, for some appropriate choice of the $\wsp$ semi-norm.
[Recall that the answer is positive when $s=1$ \cite{ls}.]
\er
%[The answer to the above question is positive when $s=1$ \cite[Theorem 3.4]{ls}.]

\medskip
We now turn to the Hausdorff distance.
\bt
\label{vp1}
Let $N\geq 1$. We have
\ben
\item
For every $1\le p<\infty$ 
\be
\label{ve50}
H-\dist_{W^{N/p,p}}({\cal E}_{d_1}, {\cal E}_{d_2})\leq C_{p, N}|d_1-d_2|^{1/p},\ \fo d_1, d_2\in\Z.
\ee
\item
%\todo{Hopefully}
If $s>0$ and $1\leq p<\infty$ are such that $sp>N$, then
\be
 \label{ve9}
 \Dist_{\wsp}({\cal E}_{d_1}, {\cal E}_{d_2})=\infty,\ \fo d_1, d_2\in\Z\text{ such that }d_1\neq d_2.
 \ee
\een
\et
%\bpr
%\label{ve4}
%There exist constants $C_1, C_2>0$ such that
%\be
%\label{ve5}
%C_1\leq \dist_{W^{N,1}}({\cal E}_{d_1}, {\cal E}_{d_2})\leq C_2,\ \fo d_1, d_2\in\Z\text{ such that }d_1\neq d_2.
%\ee
%\epr
% \todo{The last result seems reasonable. Everyone will think about it, but if no progress let as open problem.}
%\todo{For IS}
%\bpr
%\label{ve40}
%There exists a constant $C_3$ such that
%\be
%\label{ve50}
%H-\dist_{W^{N,1}}({\cal E}_{d_1}, {\cal E}_{d_2})\leq C_3|d_1-d_2|,\ \fo d_1, d_2\in\Z.
%\ee
%\epr
%\todo{For IS}
%We continue with a family of spaces playing the role of $W^{1/p, p}$ with $1<p<\infty$.
%\bpr
%\label{ve3}
%Let $0<s\leq N$ and $1<p<\infty$ be such that $sp=N$. We have $\dist_{\wsp}({\cal E}_{d_1}, {\cal E}_{d_2})=0$.
%\epr
%\bpr
%\label{ve30}
%Let $0<s< N$ and $1<p<\infty$ be such that $sp=N$. We have 
%\be
%\label{ve31}
%H-\dist_{\wsp}({\cal E}_{d_1}, {\cal E}_{d_2})\leq C_{s,p}|d_1-d_2|^{1/p},\ \fo d_1, d_2\in\Z.
%\ee
%\epr
%\br
%\label{vo5}
%As in dimension one, we do not know whether the reverse to \eqref{ve50} holds. More specifically, we do not know whether  for  $1\leq p<\infty$ there exists some $C'_{p, N}>0$ such that 
%\be
%\label{ve500}
%H-\dist_{W^{N/p,p}}({\cal E}_{d_1}, {\cal E}_{d_2})\geq C'_{p, N}|d_1-d_2|^{1/p},\ \fo d_1, d_2\in\Z.
%\ee
%\er

We conclude with three questions.
\begin{open-problem}
\label{sa2}
Is it true that for every $d_1$, $d_2$, $N$, $s$, $p$
\be
\label{Y1}
\Dist_{\wsp}({\cal E}_{d_1}, {\cal E}_{d_2})=\Dist_{\wsp}({\cal E}_{d_2}, {\cal E}_{d_1})?
\ee
(recall that $\Dist_{\wsp}({\cal E}_{d_1}, {\cal E}_{d_2})$ has been defined in \eqref{sa1}). Or even better:
\be
\label{Y2}
\text{Does }\Dist_{\wsp}({\cal E}_{d_1}, {\cal E}_{d_2})\text{ depend only on }|d_1-d_2|\text{ (and }s, p, N\text{)}?
\ee
\end{open-problem}

There are several cases where we have an explicit formula for 
 $\Dist_{\wsp}({\cal E}_{d_1}, {\cal E}_{d_2})$ and in all such cases
 \eqref{Y2} holds. See e.g. the proofs of Theorem \ref{vm1}, items
 {\it 1} and {\it 3}, and Theorem \ref{vp1}, item {\it 2}. We may also
 ask questions similar to \eqref{Y2} for $\dist_{\wsp}({\cal E}_{d_1},
 {\cal E}_{d_2})$ and for $H-\dist_{\wsp}({\cal E}_{d_1}, {\cal
   E}_{d_2})$ (assuming the answer to \eqref{Y2} is negative); again,
 the answer is positive in many cases. A striking special case still open when $N=1$ is: does $\dist_{W^{2,1}}({\cal E}_{d_1}, {\cal E}_{d_2})$ depend only on $|d_1-d_2|$? 
 
 \begin{open-problem}
 \label{sa3}
 Is it true that for every $N\ge 1$ and every $1\le p<\infty$, there exists some $C'_{p, N}>0$ such that 
 \be
\label{ve500}
\Dist_{W^{N/p,p}}({\cal E}_{d_1}, {\cal E}_{d_2})\geq C'_{p, N}|d_1-d_2|^{1/p},\ \fo d_1, d_2\in\Z?
\ee
 \end{open-problem}
 A weaker version of Open Problem \ref{sa3} is obtained when we
 replace $\Dist_{W^{N/p,p}}$ by $H-\dist{W^{N/p,p}}$ (there
 will be no difference of course in case the answer to Open Problem
 \ref{sa2} is positive):
 \setcounter{op}{1}
 \begin{op}
\label{sa3'}
 With the same assumptions as in Open Problem
 \ref{sa3}, is it true that 
 \be
\label{ve500H}
H-\dist_{W^{N/p,p}}({\cal E}_{d_1}, {\cal E}_{d_2})\geq C'_{p, N}|d_1-d_2|^{1/p},\ \fo d_1, d_2\in\Z?
\ee
 \end{op}
 The only case for which Open Problem \ref{sa3} is  settled is
 $[N=1, p=1]$  (see Theorem \ref{vm1}, item {\it 1}). We emphasize three cases of special interest:
 {\it 1.} $[N=1, p=2]$, {\it 2.} $[N=2, p=2]$ and {\it 3.} $[N=2, p=1]$. In case {\it 1}, the answer to Open Problem \ref{sa3'} is positive (see Corollary \ref{ha8}). See also Section \ref{sa4} where further partial answers are presented.

 Here is another natural open problem. Recall that for any $f\in W^{N/p,p}(\sn ; \sn)$ and any sequence $(f_n)\subset W^{N/p, p}(\sn ; \sn)$ such that $|f_n-f|_{W^{N/p, p}}\to 0$, we have $\deg f_n\to \deg f$. We also know (Theorem \ref{vo1}, item {\it 1}) that there exist sequences $(f_n)$, $(g_n)$ in $W^{N/p, p}(\sn ; \sn)$ such that $|f_n-g_n|_{W^{N/p,p}}\to 0$ but $|\deg f_n-\deg g_n|=1$, $\fo n$.

 \begin{open-problem}
\lab{tD2}
Is it true that $|\deg f_n-\deg g_n|\to 0$ for any sequences $(f_n)$, $(g_n)$ in $W^{N/p, p}(\sn ; \sn)$ such that 
\bes
|f_n-g_n|_{W^{N/p,p}}\to 0\text{ as }n\to\infty
\ees
and 
\bes
|f_n|_{W^{N/p,p}}+|g_n|_{W^{N/p, p}}\text{ remains bounded as }n\to\infty?
\ees
\end{open-problem}

%Finally, we investigate a family of spaces playing the role of $\wsp$ with $sp>1$.
%\bpr
%\label{ve6}
%Let $s>0$ and $1\leq p<\infty$ be such that $sp>N$. Then there exist constants $C'_{s, p}, C_{s, p}>0$ such that
%\be
%\label{ve7}
%C'_{s, p}\leq \dist_{\wsp}({\cal E}_{d_1}, {\cal E}_{d_2})\leq C_{s, p},\ \fo d_1, d_2\in\Z\text{ such that }d_1\neq d_2.
%\ee
%\epr
%\todo{For IS}

% \bpr
% \label{ve8}
% Let $s>0$ and $1\leq p<\infty$ be such that $sp>N$. We have
% \be
% \label{ve9}
% H-\dist_{\wsp}({\cal E}_{d_1}, {\cal E}_{d_2})=\infty,\ \fo d_1, d_2\in\Z\text{ such that }d_1\neq d_2.
% \ee
% \epr

Our paper is organized as follows. In Section~\ref{sec:ua2} we recall
some known properties of $\wsp (\sn ;
\sn)$. Sections~\ref{sec:ua5}--\ref{ua7} concern only the case $N=1$, 
while Sections~\ref{ve1}--\ref{sa4} deal with $N\geq1$. The proofs of Theorems
\ref{bgl20} and \ref{ub90} are presented in Sections \ref{sec:ua5} and
\ref{ub6}. \rth{uc3}, item {\it 1} and \rth{vm1}, items {\it 2--3},
are special cases of, respectively, \rth{vo1}, item {\it 1} and
\rth{vp1},  items {\it 1--2}; their proofs are presented in
Section~\ref{ve1}. \rth{uc3}, item {\it 2} is established in
 Section~\ref{ua7}. The proof of \rth{vm1}, item {\it 1} appears in
 Section~\ref{sec:ua5}. Theorems \ref{vo1} and \ref{vp1} belong to
 Section~\ref{ve1}. Partial solutions to the open problems are given
 in Section~\ref{sa4}. A final Appendix gathers various auxiliary results.

\subsubsection*{Acknowledgments} The
first author (HB) was partially supported by NSF grant DMS-1207793.  The second author (PM) was partially  supported  by the LABEX MILYON (ANR-10-LABX-0070) of Universit\'e de Lyon,
within the program \enquote{Investissements d'Avenir} (ANR-11-IDEX-0007) operated
by the French National Research Agency (ANR). The third author (IS)
was supported by  the Israel Science Foundation (Grant No. 999/13).

\tableofcontents

\section{Some standard properties of maps $f:\sn\to\sn$}
\label{sec:ua2}

In this section, we alway assume that \eqref{ua1} holds. 
\bl
\label{ih1}
$C^\infty(\sn ; \sn)$ is dense in $\wsp (\sn ; \sn)$.
\el

When $s=1$, $p=2$, $N=2$, the above was proved in \cite{su}. The argument there extends to the general case.

%When $1<p<\infty$ and $s=N/p$, Lemma \ref{ih1} can be complemented as follows.

When
\be
\label{ff1}
[0\le s-N/p<1]\text{ or }[s- N/p=1\text{ and }p>1],
\ee
we can complement Lemma \ref{ih1} as follows.
\bl
\label{ih2}
Assume that \eqref{ff1} holds. Then every map $f\in W^{s, p}(\sn ; \sn)$ can be approximated by a sequence $(f_n)\subset C^\infty(\sn ; \sn)$ such that every $f_n$ is constant near some point.
\el

We note that condition \eqref{ff1} is equivalent to \eqref{ua1} + the non embedding $\wsp\not\hookrightarrow C^1$. The non embedding is also necessary for the validity of the conclusion of Lemma \ref{ih2}. Indeed, a $C^1$ function $f$, say on the real line, whose derivative does not vanish, cannot be approximated in $C^1$ by a sequence $(f_n)$ such that each $f_n$ is constant near some point.

 The proof of Lemma \ref{ih2} is postponed to the Appendix. 

%
%The above is obtained combining Lemma \ref{ih1} with the fact that the $W^{N/p, p}$ capacity of a point is zero.

\begin{theo}[\cite{brezisnirenberg1}]
\label{ih3} 
{\it For $1\le p<\infty$, the degree of smooth maps $f:\sn\to\sn$ is continuous with respect to the $W^{N/p, p}$ convergence.

As a consequence, under assumption \eqref{ua1} the degree extends to maps in $W^{N/p,p} (\sn ; \sn)$. Moreover,  if $(f_n)$ and $f$ are in $W^{N/p, p}$ and $|f_n-f|_{W^{N/p, p}}\to 0$, then $\deg f_n\to \deg f$.}
\end{theo}

This follows from the corresponding assertion for the BMO convergence \cite{brezisnirenberg1} and the fact that $W^{N/p, p}\hookrightarrow\, $BMO.

% As a consequence of Theorem \ref{ih3}, under the assumption  \eqref{ua1}, the degree extends to $\wsp (\sn ; \sn)$ maps. The next result is a clear consequence of Theorem \ref{ih3} and of the Sobolev embedding
% \be
% \label{za2}
% \wsp\hookrightarrow C^0\ \text{if }[s>0\text{ and }sp>N]\text{ or }[s= N\text{ and }p=1].
% \ee
% 
% \bl
% \label{za1}
% If $[sp>N]$ or $[s=N\text{ and } p=1]$, then the degree of $\wsp (\sn ; \sn)$ maps coincides with the usual one of continuous maps.
% \el

When $N=1$, an alternative equivalent definition of the degree can be obtained via lifting \cite{bmp, bm}. In this case, given $f\in \wsp(\so ; \so)$, it is always possible to write 
\be
\label{ii1}
f(e^{\im\theta})=e^{\im\varphi(\theta)},\ \fo \theta\in\R,\text{ for some }\varphi\in \wsp_{loc}(\R ; \R)
\ee
(no condition on $s$ and $p$ \cite{lss}). 

If, in addition, \eqref{ua1} holds, then the function $\varphi(\cdot+2\pi)-\varphi(\cdot)$ is constant a.e. \cite{lss}, and we have 
\be
\label{ii3}
\deg f=\frac 1{2\pi} (\varphi(\cdot+2\pi)-\varphi(\cdot)).
\ee

If instead of \eqref{ua1} we assume that either [$sp>1$] or [$s=1$ and $p=1$], then $\varphi$ is continuous and \eqref{ii3} becomes
\be
\label{ii4}
\deg f=\frac 1{2\pi} (\varphi(2\pi)-\varphi(0))=\frac 1{2\pi} (\varphi(\pi)-\varphi(-\pi)).
\ee

Finally, we mention the formula
\be
\label{ii5}
\deg f=\frac 1{2\pi}\int_{\so}f\wedge\dot f,\ \fo f\in W^{1,1}(\so ; \so).
\ee
\section{$W^{1,1}$ maps}
\label{sec:ua5}

\begin{proof}[Proof of Theorem \ref{bgl20} for $p=1$, and  Theorem \ref{ub90}, item 1]
${}$\\
{\it Step 1.} Proof of \enquote{$\leq$} in \eqref{1a1}\\
With no loss of generality we may assume that  $d_1>d_2$ and $d_1>0$. Set $d:=d_1-d_2$ and $L:=d+1$. We define $f (e^{\im\theta}):=e^{\im\va(\theta)}\in{\cal E}_{d_1}$, $g (e^{\im\theta}):=e^{\im\psi(\theta)}\in{\cal E}_{d_2}$,  where $\va, \psi\in W^{1,1}((0,2\pi))$ are defined as follows:
   \begin{equation*}
  \va(\theta):=\begin{cases}
L\theta, &\text{if } \theta\in[0,2d\,\pi /L)\\
 Ld_2\theta+2(d_1-Ld_2 )\,\pi ,&\text{if }\theta\in[2d\, \pi /L,2\,\pi)
\end{cases},
   \end{equation*}
and 
\bes
\psi(\theta):=
\begin{cases}
L\, \dist (\theta, 2\,\pi\,\Z/L),&\text{if }\theta\in[0,2d\, \pi /L)\\
\varphi(\theta)-2d\,\pi,&\text{if }\theta\in[2d\, \pi /L,2\,\pi)
\end{cases}
   \end{equation*}
 (and thus on $[0,2d\, \pi/L]$ the graph of $\psi$ is a zigzag consisting of $d$ triangles). 
   
%$f_1(e^{\im \theta})=e^{\im \phi(\theta)}$ where
%   $\phi$ is defined as follows. First, for $j=1,\ldots,m$,
%   \begin{equation*}
%     \phi(\theta)=\begin{cases}\d
%            L\theta,&\d\text{if }
%            \theta\in\left[(j-1)\frac{2\pi}{L},(j-1)\frac{2\pi}{L}+\frac{\pi}{L}\right]\\
%     \d     -L\theta  & \d\text{if }
%           \theta\in\left [(j-1)\frac{2\pi}{L}+\frac{\pi}{L},j\frac{2\pi}{L}\right]
%            \end{cases},
%      \end{equation*}
% and then $e^{\im \phi(\theta)}=f_2(e^{\im \theta})$ for $\theta\in[\frac{2\pi m}{L},2\pi]$.

%We next make the crucial observation that $f_1-f_2$ takes only imaginary values, and therefore, with $h:=|f_1-f_2|$, we have 
%\be
%\label{hn1}
%|\dot f_1-\dot f_2|=|\dot h|\ \text{a.e.}
%\ee
%%Note that
%%\bes
%%\int_{\so}|\dot f_1-\dot f_2|=\int_0^{2\pi m/L}|(e^{\im\phi_1}-e^{\im\phi_2})'|.
%%\ees
%
% Since, on $[0, 2\pi m/L]$,  the function $k(\theta)=h(e^{\im\theta})=|f(e^{\im\theta})-g(e^{\im\theta})|$ increases from $0$ to $2$
% and then decreases back from $2$ to $0$  $m$ times, we obtain from \eqref{hn1} that 
% \begin{equation*}
%   \int_{{\mathbb S}^1} |\dot{f_1}-\dot{f_2}|=\int_0^{2\pi m/L} |\dot k(\theta)|\, d\theta= 4m=4(d_1-d_2). \end{equation*}
   
   For $k\in\Z$, $0\le k\le d-1$, set
   \bes
   I_k=\left[\frac{2k\,\pi}L,\, \frac{(2k+1)\, \pi}L  \right]\ \text{and}\ J_k=\left[\frac{(2k+1)\,\pi}L,\, \frac{(2k+2)\, \pi}L  \right].
   \ees
   
   Note that
   \bes
   \psi(\theta)=\begin{cases}
   L\theta-2k\, \pi,&\text{if }\theta\in I_k\\
  2(k+1)\, \pi- L\theta,&\text{if }\theta\in J_k
   \end{cases},
   \ees
   so that $g=f$ on $I_k$ and $g=\overline f$ on $J_k$. Hence
  \bes
  \left|\dot f-\dot g\right|=\begin{cases}
  0,&\text{on }I_k\\
  -2\, (\sin\va)\, \va',&\text{on }J_k
  \end{cases}.
  \ees
  
  Therefore
  \bes
  \int_{\so}\left|\dot f-\dot g\right|=2\sum_{k=0}^{d-1}\int_{J_k} (\cos \va)'(\theta)\, d\theta=4\, d= 4\, (d_1-d_2).
  \ees
  
\smallskip
\noindent
{\it Step 2.} Proof of \enquote{$\geq$} in \eqref{1a1}\\
We may assume that $d:=d_1-d_2>0$. We prove that  when $f\in{\cal E}_{d_1}$ and $g\in{\cal E}_{d_2}$ we have $\int_{\so}|\dot f-\dot g|\geq 4 d$. The map $f/g$ is onto (since its degree is $d\neq 0$), and thus with no loss of generality we may assume that $f(1)=g(1)$. Write $f(e^{\im\theta})=e^{\im\va(\theta)}\, g(e^{\im\theta})$, with $\va\in W^{1,1}((0,2\pi))$. We have $\va(2\pi)-\va(0)=2d\, \pi$, and we may assume that $\va(0)=0$. Consider $0=t_0<\tau_0<t_1<\cdots<\tau_{d-1}<t_d= 2\pi$ such that $\va(t_j)=2\pi j$, $j=0,\ldots, d$,  and $\va(\tau_j)=2\pi j+\pi$, $j=0,\ldots, d-1$.  Thus the function $w:=|f-g|$ satisfies $w(e^{\im t_j})=0$ and $w(e^{\im \tau_j})=2$. Therefore, we have $\int_{\so}|\dot w|\geq 4d$. In order to conclude, it suffices to note the inequality $|\dot w|\leq |\dot f-\dot g|$ a.e. 
\end{proof}

We now turn to the properties of the Hausdorff distance in $W^{1,1}$.
\begin{proof}[Proof of Theorem \ref{vm1}, item 1]
% In fact we will prove a stronger property, namely 
% \be
% \label{XXX}
% \Dist_{W^{1,1}}({\cal E}_{d_1}, {\cal E}_{d_2})=2\pi\, |d_1-d_2|,\ \fo d_1, d_2\in\Z.
% \ee

\smallskip
\noindent
{\it Step 1.} Proof of \enquote{$\le$} in \eqref{XX}   \\
By symmetry, it suffices to prove that for every  $f\in\mathcal{E}_{d_1}$ and every $\ve>0$ there exists some 
$g\in\mathcal{E}_{d_2}$ satisfying
\begin{equation}
  \label{eq:43}
\int_{{\mathbb S}^1} |\dot{f}-\dot{g}|\leq 2\pi|d_1-d_2|+\ve.
\end{equation}

 By density of $C^\infty({\mathbb S}^1;{\mathbb S}^1)$ in  $W^{1,1}({\mathbb S}^1;{\mathbb S}^1)$ it suffices
 to prove \eqref{eq:43} for smooth $f$. Moreover,  we may assume that $f$ is constant near some point, say $1$ (see Lemma \ref{ih2}). We may thus write 
$
f(e^{\im \theta})=e^{\im \va(\theta)}$, $\theta\in [0,2\pi]$,
for some smooth $\va$ satisfying $\va(2\pi)-\va(0)=2 \pi\, d_1$ and constant near $0$. For a small $\lambda>0$ define $\psi=\psi^{(\lambda)}$ on $[0,2\pi]$ by 
\begin{equation}
\label{eq:47}
  \psi(\theta):=\begin{cases}
\d\va(\theta)-\frac{2d\, \pi}{\lambda}\,\theta,&\text{if }
\theta\in[0,\lambda]\\
 \va(\theta)-2d\, \pi, &\text{if } \theta\in(\lambda,2\pi]
\end{cases}
\end{equation}
 (where $d:=d_1-d_2$), and then set
 $g(e^{\im \theta}):=e^{\im \psi(\theta)}\in\mathcal{E}_{d_2}$. Clearly,
 \begin{equation*}
%\label{eq:46}
   \int_{{\mathbb S}^1} |\dot{f}-\dot{g}|=\int_0^\lambda
   \left|(e^{\im \psi}-e^{\im \va})'\right|=2|d|\, \pi=2\pi |d_1-d_2|.
 \end{equation*}

\smallskip
\noindent
{\it Step 2.} Proof of 
%item {\it 2}\\
%We may assume that $0\leq d_2<d_1$. Let $g(z)=z^{d_2}$. It suffices to  prove that $|f-g|_{W^{1,1}}\geq 2\pi (d_1-d_2)$, $\fo f\in {\cal E}_{d_1}$. 
% By the
%triangle inequality, for any such $f$ we have
%\be
%\label{uh1}
%  \int_{{\mathbb S}^1}|\dot{f}-\dot{g}|\geq
%  \int_{{\mathbb S}^1}[|\dot{f}|-|\dot{g}|]\geq \left|\int_{{\mathbb S}^1}f\land\dot{f}\right|-2\pi  d_2= 2\pi\, (d_1-d_2).\qedhere
%\ee
\be
\label{Z1}
\Dist_{W^{1,1}}({\cal E}_{d_1}, {\cal E}_{d_2})\ge 2\pi\, |d_1-d_2|,\ \fo d_1, d_2\text{ with }0\le d_1<d_2.
\ee
Let $f(z):=z^{d_1}\in {\cal E}_{d_1}$. It suffices to prove that 
\bes
|f-g|_{W^{1,1}}\ge 2\pi\, (d_2-d_1),\ \fo g\in {\cal E}_{d_2}.
\ees

By the triangle inequality, for any such $g$, we have
\be
\label{uh1}
  \int_{{\mathbb S}^1}|\dot{f}-\dot{g}|\geq
  \int_{{\mathbb S}^1}[|\dot{g}|-|\dot{f}|]\geq \left|\int_{{\mathbb S}^1}g\land\dot{g}\right|-2\pi  d_1= 2\pi\, (|d_2|-d_1)=2\pi\, (d_2-d_1),
\ee
since $|\dot f|=d_1$ on $\so$.

\smallskip
\noindent
{\it Step 3.} Proof of
\be
\label{Z3}
\Dist_{W^{1,1}}({\cal E}_{d_1}, {\cal E}_{d_2})\ge 2\pi\, |d_1-d_2|,\ \fo d_1\ge 0,\ \fo d_2\in\Z\text{ with }d_2<d_1.
\ee
The case $d_1=0$ is trivial since we may take as above $f(z):=z^0=1$ and apply \eqref{uh1}. 

We now turn to the case $d_1>0$ and $d_2<d_1$ which is quite involved. Inequality \eqref{Z3} is a direct consequence of the following 
\bl
\label{sa6}
Assume that $d_1>0$ and $d_2<d_1$. Then for each $\delta>0$ there exists $f\in {\cal E}_{d_1}$ such that 
\be
\label{sa7}
\int_{\so}|\dot f-\dot g|\ge (2\pi-\delta)\, (d_1-d_2),\ \fo g\in {\cal E}_{d_2}.
\ee
\el

\begin{proof}
For large $n$ (to be chosen later) let $f_n(e^{\im \theta})=e^{\im
  \varphi_n(\theta)}\in\mathcal{E}_{d_1}$, with $\varphi_n\in W^{1,1}((0,2\pi))$
defined by setting $\varphi_n(0)=0$ and 
\begin{equation}
\label{eq:phi-prime}
  \varphi_n'(\theta)= \begin{cases} d_1n,&
  \theta\in[2j\, \pi/n^2), (2j+1)\, \pi/n^2]\\
             -d_1(n-2),&
             \theta\in      ((2j+1)\, \pi/n^2), (2j+2)\, \pi/n^2]
  \end{cases},\ j=0,1,\ldots,n^2-1. 
\end{equation}

Therefore, the graph of $\varphi_n$ is a zigzag 
of $n^2$ triangles. Note that the average gradient of $\varphi_n$ is
$d_1$, since
\begin{equation}
\label{eq:small}
  \int_{2j\, \pi/n^2}^{(2j+2)\,\pi/n^2}\varphi_n'=2\pi \frac{d_1}{n^2},\ j=0,1,\ldots,n^2-1.
\end{equation}
Hence $\int_0^{2\pi}\varphi_n'=2\pi\, d_1$ (so indeed $f_n\in
\mathcal{E}_{d_1}$). On the other hand, note that
\begin{equation*}
  \int_{2j\, \pi/n^2}^{(2j+2)\, \pi/n^2)}|\varphi_n'|=2(n-1)\, \pi 
  \frac{d_1}{n^2},\ j=0,1,\ldots, n^2-1\;\Longrightarrow\;\int_0^{2\pi}
  |\varphi_n'|=2(n-1)\, \pi\, d_1 ,
\end{equation*}
i.e., $\lim_{n\to\infty} \|\dot f_n\|_{L^1(\so)}=\infty$.

Consider now any $g\in\mathcal{E}_{d_2}$ and write $g(e^{\im
  \theta})=e^{\im\psi(\theta)}$ with 
$\psi\in W^{1,1}((0,2\pi))$ satisfying $\psi(2\pi)-\psi(0)=2\pi\,
d_2$. For convenience we extend both $\varphi_n$ and $\psi$ to all of $\R$ in such a way that
the extensions are continuous functions whose derivatives are $2\pi$-periodic. 
Set $h=f_n\, \overline g\in\mathcal{E}_d$ with $d:=d_1-d_2>0$. Hence, $h(e^{\im
  \theta})=e^{\im\eta(\theta)}$ with $\eta:=\varphi_n-\psi$. We can find $d$
(closed) arcs on $\so$, $I_1,\ldots,I_d$, with disjoint interiors such that:
\begin{equation*}
  I_j=\{e^{\im \theta};\,\theta\in[s_j,t_j]\},\ h(e^{\im s_j})=h(e^{\im
    t_j})=1\text{ and }\int_{s_j}^{t_j}\eta'=2\pi,\text{ for }j=1,\ldots,d.
\end{equation*}
For small $\varepsilon>0$ define,  for each $j=1,\ldots,d$:
\begin{equation}
\label{eq:alpha-beta}
\begin{array}{ll}
  \alpha^{-}_j=\max\left\{\theta\in[s_j,t_j];\, h(e^{\im\theta})=e^{\im\varepsilon}\right\}, 
&\beta^{-}_j=\min\left\{\theta\in[\alpha^{-}_j,t_j]\,;\,h(e^{\im\theta})=e^{\im(\pi-\varepsilon)}\right\},\\
  \alpha^{+}_j=\max\left\{\theta\in[\beta^{-}_j,t_j];\,h(e^{\im\theta})=e^{\im(\pi+\varepsilon)}\right\},\ 
& \beta^{+}_j=\min\left\{\theta\in[\alpha^{+}_j,t_j];\,h(e^{\im\theta})=e^{\im (2\pi-\varepsilon)}\right\}.
\end{array}
\end{equation}
Then, set
$ I_j^{\pm}:=\{e^{\im
    \theta};\,\theta\in[\alpha_j^{\pm},\beta_j^{\pm}]\}
$.
Using the equality
\bes
f_n-g=e^{\im\varphi_n}-e^{\im\psi}=2\im\, \sin\left(\frac{\varphi_n-\psi}{2}\right)\, e^{\im\, (\varphi_n+\psi)/2},
\ees
we obtain
\begin{equation}
  \label{eq:gradf-g}
|\dot{f}_n-\dot{g}|^2=\cos^2\left(\frac{\varphi_n-\psi}{2}\right)\, (\varphi_n'-\psi')^2+\sin^2\left(\frac{\varphi_n-\psi}{2}\right)\,  (\varphi_n'+\psi')^2.
\end{equation}
Note that by the definition of $I_j^\pm$ we have
\begin{equation}
\label{sa9}
  z=e^{\im\theta}\in
  I_j^\pm\;\Longrightarrow\;\left|\sin\left(\frac{\varphi_n(\theta)-\psi(\theta)}{2}\right)\right|,\, \left|\cos\left(\frac{\varphi_n(\theta)-\psi(\theta)}{2}\right)\right|\geq \sin(\varepsilon/2).
\end{equation}
Combining \eqref{sa9} with \eqref{eq:gradf-g} and \eqref{eq:phi-prime} gives
\begin{equation}
\label{eq:integral}
\begin{aligned}
  \int_{ I_j^\pm}|\dot{f}_n-\dot{g}|&\geq \sin(\varepsilon/2) \int_{\alpha_j^{\pm}}^{\beta_j^{\pm}}
 \sqrt{(\varphi_n'-\psi')^2+(\varphi_n'+\psi')^2}\\
 &\geq
\sqrt{2} \sin(\varepsilon/2) 
\int_{\alpha_j^{\pm}}^{\beta_j^{\pm}}|\varphi_n'|\geq \sqrt{2} \sin(\varepsilon/2)\, d_1(n-2)\, | I_j^\pm|,
\end{aligned}
\end{equation}
where $| I_j^\pm|:=\beta_j^{\pm}-\alpha_j^{\pm}$.
If for one of the arcs $I_j^\pm$ there holds
\begin{equation*}
  \sqrt{2} \sin(\varepsilon/2)\, d_1(n-2)\, | I_j^\pm|>2\pi d,
\end{equation*}
then we clearly have $\int_{\so}|\dot{f}-\dot{g}|>2\pi d$ by
\eqref{eq:integral}, and \eqref{sa7} follows. Therefore, we are
left with the case where
\begin{equation}
  \label{eq:case}
  |I_j^{-}|,|I_j^{+}|\leq\frac{c_0}{n},\ j=1,\ldots,d,
\end{equation}
 where $c_0=c_0(d_1,d_2,\varepsilon)$.

While in the previous case the lower bound followed from the fact that
$|\varphi_n'|$ is large (i.e., of the order of $n$), the argument under
assumption \eqref{eq:case} uses another property of
$\varphi_n$. Namely, thanks to \eqref{eq:small}, the change of
$\varphi_n$ on an interval of length $O(1/n)$ (like is the case for
$I_j^{\pm}$) is only of the order $O(1/n)$. It follows that $f_n$ is
``almost'' a constant on the corresponding arc and an important contribution to
the BV norm of $f_n-g$ comes from the change of the phase $\psi$ on
the corresponding interval. The latter equals approximately $\pi-2\varepsilon$, and
summing the contributions from all the arcs yields the desired lower
bound. 
The details are given below. 

 In the sequel we will denote by $c$
 different constants depending on $d_1,d_2$ and $\varepsilon$ alone.
A direct consequence of \eqref{eq:small} that will play a key role in
the sequel is the following:
\begin{equation}
  \label{eq:moderate}
 \left|\int_J \varphi_n'\right|\leq \frac{c}{n},\ \text{for every interval }J\subset\R\text{ with }|J|\le \frac{c_0}n.
  \end{equation}
  
 This implies that 
\begin{equation*}
   |f_n(z_1)-f_n(z_2)|\leq\frac{c}{n},\ \fo z_1,z_2\in I_j^{\pm},\ j=1,\ldots,d.
\end{equation*}

Therefore, for each $I_j^{\pm}$ there exists $\nu_j^{\pm}\in\so$ such
that
\begin{equation}
  \label{eq:def-nu}
   |f_n(z)-\nu_j^{\pm}|\leq \frac{c}{n},\ \fo z\in I_j^{\pm},\ j=1,\ldots,d.
\end{equation}

 By \eqref{eq:def-nu} we
have
\begin{equation}
\label{eq:norm1}
  \left|1-|g(z)-(f_n(z)-\nu_j^{\pm})|\right|\leq\frac{c}{n},\ \fo z\in I_j^{\pm},\ j=1,\ldots,d.\end{equation}

Fix an arc $I_j^\pm$. By \eqref{eq:norm1},  we can define on $[\alpha_j^{\pm},\beta_j^{\pm}]$ a
 $W^{1,1}$-function $\psi_n=\psi_{n, j, \pm}$, determined uniquely up to addition of an
 integer multiple of $2\pi$, by
 \begin{equation}
   \label{eq:psin}
g(e^{\im\theta})-(f_n(e^{\im\theta})-\nu_j^{\pm})=|g(e^{\im\theta})-(f_n(e^{\im\theta})-\nu_j^{\pm})|\, e^{\im\psi_n(\theta)}.
 \end{equation}
 
 From \eqref{eq:def-nu}--\eqref{eq:psin} we have
 \begin{equation}
   \label{eq:g-psin}
|e^{\im\psi(\theta)}- e^{\im\psi_n(\theta)}|\leq \frac{c}{n},\ \fo \theta\in [\alpha_j^{\pm},\beta_j^{\pm}],
 \end{equation}
 and 
 \be
 \label{sa10}
 |\dot g (e^{\im \theta})-\dot f_n (e^{\im\theta})|\ge |g(e^{\im\theta})-(f_n(e^{\im\theta})-\nu_j^{\pm})|\, |\psi_n'(\theta)|\ge \left( 1-\frac cn \right)\, |\psi_n'(\theta)|.
 \ee
 
 By \eqref{sa10}, we have
 \begin{equation}
\label{eq:gf-n}
    \int_{ I_j^\pm}|\dot g-\dot f_n|\geq
    \left(1-\frac{c}{n}\right)\int_{\alpha_j^{\pm}}^{\beta_j^{\pm}}|\psi_n'|\geq  \left(1-\frac{c}{n}\right)|\psi_n(\beta_j^{\pm})-\psi_n(\alpha_j^{\pm})|.
 \end{equation}
By \eqref{eq:g-psin}, \eqref{eq:gf-n},\eqref{eq:moderate} and
\eqref{eq:alpha-beta}, we obtain
\begin{equation}
  \label{eq:lb}
   \int_{ I_j^\pm}|\dot g-\dot f_n|\geq
   \left(1-\frac{c}{n}\right)|\psi(\beta_j^{\pm})-\psi(\alpha_j^{\pm})|-\frac{c}{n}\geq
   \big(1-\frac{c}{n}\big)|\eta(\beta_j^{\pm})-\eta(\alpha_j^{\pm})|-\frac{c}{n}\geq \left(1-\frac{c}{n}\right)(\pi-2\varepsilon).
\end{equation}
Summing \eqref{eq:lb} over  the $2d$ arcs 
$I_j^{-}$, $I_j^{+}$, $j=1,\ldots, d$ yields
\begin{equation}
\label{sa11}
  \int_{ I_j^\pm}|\dot g-\dot f_n|\geq  \left(1-\frac{c}{n}\right)\, (2\pi\, d-4\varepsilon\, d).
\end{equation}

Finally we choose $\ve=\delta/8$ and $n\ge \d\frac {4\pi}\delta\, c(d_1, d_2, \ve)$ and deduce from \eqref{sa11} that 
\eqref{sa7} holds.
\end{proof}
\smallskip
\noindent
{\it Step 4.} Proof of \eqref{XX} completed\\
Combining Steps 1, 2 and 3 we find that
\bes
\Dist_{W^{1,1}}({\cal E}_{d_1}, {\cal E}_{d_2})=2\pi\, |d_1-d_2|,\ \fo d_1\ge 0,\  \fo d_2\in\Z,
\ees
%and since $\varepsilon$ can be chosen arbitrary small, \eqref{Z3} follows provided we choose $n\geq n_0(d_1,d_2,\varepsilon)$.
which yields directly
\bes
\Dist_{W^{1,1}}({\cal E}_{d_1}, {\cal E}_{d_2})=2\pi\, |d_1-d_2|,\ \fo d_1\in\Z,\ \fo d_2\in\Z. \qedhere
\ees

\end{proof}
We close this section with some results concerning the attainability of
$\Dist_{W^{1,1}}({\cal E}_{d_1}, {\cal E}_{d_2})$. For any $d_1\neq
d_2$ we may ask (question 1) whether there exists $f\in{\cal E}_{d_1}$ such that 
\begin{equation}
\label{eq:att}
  d_{W^{1,1}}(f,\mathcal{E}_{d_2}):=\inf_{g\in\mathcal{E}_{d_2}}|f-g|_{W^{1,1}}=\Dist_{W^{1,1}}({\cal E}_{d_1}, {\cal E}_{d_2})\,,
\end{equation}
and in case the answer to question 1 is positive for some
$f\in\mathcal{E}_{d_1}$, we may ask (question 2) whether 
 the infimum $\inf_{g\in\mathcal{E}_{d_2}}|f-g|_{W^{1,1}} $ is actually a
minimum, i.e., for some $g\in\mathcal{E}_{d_2}$,
\begin{equation}
\label{eq:strong-att}
  |f-g|_{W^{1,1}}=d_{W^{1,1}}(f,\mathcal{E}_{d_2})=\Dist_{W^{1,1}}({\cal E}_{d_1}, {\cal E}_{d_2})\,\,.
\end{equation}
 There is a trivial case where the answer to both questions is
affirmative, namely, when $0=d_1\neq d_2$. Indeed, for $f=1$ and
$g(z)=z^{d_2}$ we clearly have,
\begin{equation*}
  |f-g|_{W^{1,1}}=\int_{S^1}|\dot g|=2\pi|d_2|=\Dist_{W^{1,1}}({\cal E}_{0}, {\cal E}_{d_2})\,.
\end{equation*}
The next proposition provides answers to these attainability
questions, demonstrating different behaviors according to the sign of $d_1(d_2-d_1)$.
\begin{proposition}
  \label{prop:attain}
 We have
\ben
\item
If $d_1(d_2-d_1)>0$ then $f\in \mathcal{E}_{d_1}$ satisfies
\eqref{eq:att} if and only if
\begin{equation}
  \label{eq:wedge}
d_1(f\land \dot{f})\geq 0 ~\text{ a.e. in }\so\,.
\end{equation}
Among all maps satisfying \eqref{eq:att}, some  satisfy
\eqref{eq:strong-att} and others do not. 
\item
 If $d_1(d_2-d_1)<0$ then for every $f\in\mathcal{E}_{d_1}$ we have
 $d_{W^{1,1}}(f,\mathcal{E}_{d_2})<\Dist_{W^{1,1}}({\cal E}_{d_1},
 {\cal E}_{d_2})$, so \eqref{eq:att} is never satisfied.
\een
\end{proposition}
The proof relies on several lemmas. 
\bl
\label{vc2}
Let $d_1, d_2\in\Z$ be such that $d_1(d_2-d_1)>0$. If $f\in{\cal
  E}_{d_1}$ satisfies \eqref{eq:wedge} then
\begin{equation}
  \label{weak}
 \int_{{\mathbb S}^1}|\dot{f}-\dot{g}|\geq 2\pi|d_1-d_2|,\ \forall\, g\in{\cal E}_{d_2}.
\end{equation}
 If the stronger condition 
 \begin{equation}
\label{eq:stric}
   d_1 (f\wedge \dot{f})> 0 \text{ a.e. in } \so,
 \end{equation}
 holds, then
\begin{equation}
  \label{eq:str}
 \int_{{\mathbb S}^1}|\dot{f}-\dot{g}|>2\pi|d_1-d_2|,\ \forall\, g\in{\cal E}_{d_2}.
\end{equation}
\el
\begin{proof}[Proof of Lemma \ref{vc2}]
It suffices to consider the case $0<d_1<d_2$.
 Note that \eqref{eq:wedge} is equivalent to
 $\int_{\so}|\dot{f}|=\int_0^{2\pi}f\land\dot{f}=2\pi d_1$, i.e., to $f$
 being a minimizer for $\int_{\so} |v'|$ over
 $\mathcal{E}_{d_1}$ (\eqref{eq:2} for $p=1$). Therefore the same
 computation as in \eqref{uh1} yields \eqref{weak}.

Next assume the stronger condition \eqref{eq:stric}. Writing $f(e^{\im
  \theta})=e^{\im\varphi(\theta)}$, with $\varphi\in
W^{1,1}((0,2\pi))$, we then
have  $\varphi'>0$ a.e. in $(0,2\pi)$. Suppose by
contradiction that for
some $g\in{\cal E}_{d_2}$ equality holds in \eqref{weak}. 
Then \eqref{uh1} yields
\begin{equation}
\label{dot}
  |\dot g-\dot f|=|\dot g|-|\dot f|\,,~\text{a.e. in }\so\,.
\end{equation}
Writing  $g(e^{\im\theta})=e^{\im\psi(\theta)}$, with $\psi\in
W^{1,1}((0,2\pi))$, the same computation as in \eqref{eq:gradf-g}, gives
\begin{equation}
  \label{eq:52}
\left|(e^{\im \psi}-e^{\im \varphi})'\right| ^2=
\cos^2\left(\frac{\varphi-\psi}{2}\right)\, (\varphi'-\psi')^2+\sin^2\left(\frac{\varphi-\psi}{2}\right)\,  (\varphi'+\psi')^2.
\end{equation}
Combining \eqref{dot} with \eqref{eq:52} leads to 
\begin{equation}
  \label{same-comp}
\sin^2\left(\frac{\psi-\varphi}{2}\right)\, (\psi'-\varphi')^2=\sin^2\left(\frac{\psi-\varphi}{2}\right)\, (\psi'+\varphi')^2\,.
\end{equation}
 The equality \eqref{same-comp} clearly implies that $\varphi'=0$
 a.e.~on the set $\{f\neq g\}$. Since this set has positive measure,
 we reached a contradiction to \eqref{eq:stric}.
\end{proof}
\begin{lemma}
\label{ic}
If  $d_1(d_2-d_1)<0$ then
for every $f\in{\cal E}_{d_1}$ there exists $g\in{\cal
  E}_{d_2}$ such that
\begin{equation}
  \label{eq:49}
 \int_{{\mathbb S}^1}|\dot{f}-\dot{g}|<2\pi|d_1-d_2|.
\end{equation}
\end{lemma}

  The proof of Lemma \ref{ic} is quite involved. It
  is inspired by the work of H. Brezis and J.-M. Coron (see \cite{bc,b-L}) in a
  two-dimensional setting, where the importance of a strict inequality
  like \eqref{eq:49} 
   was emphasized. The heart of the estimate is the following lemma.
   \begin{lemma}
     \label{lem:strict}
   Consider any  $f\in {\cal E}_{d_1}$ and a point $\zeta\in {\mathbb S}^1$, which is 
   a Lebesgue point of $\dot{f}$ with  $(f\land
   \dot{f})(\zeta)\neq 0$. Then for every $d_2$ such that
   \begin{equation}
     \label{eq:67}
    (d_2-d_1)\cdot (f\land\dot{f})(\zeta)<0
   \end{equation}
  there exists  $g\in {\cal E}_{d_2}$ satisfying \eqref{eq:49}.
   \end{lemma}
\begin{proof}[Proof of Lemma \ref{lem:strict}]
We may assume that condition \eqref{eq:67} is satisfied
by $\zeta=1$.
 Write 
$f(e^{\im \theta})=e^{\im\varphi(\theta)}$ with $\varphi\in
W^{1,1}((0,2\pi))$ satisfying $\varphi(2\pi)-\varphi(0)=2\pi
d_1$.  By assumption, $\theta_0=0$ is a Lebesgue point of $\varphi'=f\land
   \dot{f}$ with 
 $\varphi'(0):=\alpha\neq 0$ and we have
\begin{equation}
\label{eq:48}
\lim_{\delta\to 0}\frac{1}{\delta}\int_0^\delta |  \varphi'-\alpha|=0.
\end{equation}
Denote $d=d_1-d_2$ and note that, by \eqref{eq:67}, we have $\alpha d>0$. 
For each small $\varepsilon>0$ set $g=e^{\im\psi}$, where $\psi=\psi^\ve$ is defined by
\bes
\psi(\theta)=\begin{cases}
\varphi(\theta)-\d\frac{2d\, \pi}{\ve}\, \theta,&\text{if }\theta\in [0,\ve]\\
\varphi(\theta)-2d\, \pi,&\text{if }\theta\in [\ve, 2\, \pi]
\end{cases}.
\ees

%define $\psi_2=\psi_2^{(\varepsilon)}$
%by \eqref{eq:47} and  set $f_2=e^{\im \psi_2}$. 

By
\eqref{eq:52}, we have
\be
  \label{eq:42}
   \int_{{\mathbb S}^1}|\dot{g}-\dot{f}|
%   = \int_{\{\zeta=e^{\im \theta}\,;\,\theta\in[0,\varepsilon]\}}|\dot{g}-\dot{f}|\\
   =\left(\frac{2|d|\, \pi }{\varepsilon}\right)\int_0^\varepsilon h(\theta)\, d\theta,
\ee
where
\be
\label{eq:42*}
h(\theta)=h_\ve(\theta):=\left[1+4\sin^2\left(\frac{d\, \pi\,\theta}{\varepsilon}\right)\left\{-\frac{\ve\varphi'(\theta)}{2d\, \pi}+\left(\frac{\ve\varphi'(\theta)}{2d\, \pi}\right)^2\right\}\right]^{1/2}.
\ee

Set $F:=\varphi'-\alpha$ and write
\be
\label{ia2}
(h_\ve(\theta))^2=X_\ve+Y_\ve+Z_\ve,
\ee
where
\be
\label{ia3}
X_\ve=X_\ve(\theta):=1-\frac{2\ve\alpha}{d\, \pi}\left(1-\frac{\ve\alpha}{2 d\, \pi}\right)\sin^2\, \left(\frac{d\, \pi\, \theta}{\ve}\right)= 1- \frac{2\ve\alpha}{d\, \pi}\sin^2\, \left(\frac{d\, \pi\, \theta}{\ve}\right)+ O(\ve^2),
\ee
\be
\label{ia4}
Y_\ve=Y_\ve(\theta):=\frac{2\ve F}{d\, \pi}\left(-1+\frac{\ve\alpha}{d\, \pi}\right)\sin^2\, \left(\frac{d\, \pi\, \theta}{\ve}\right)=O(\ve F),
\ee
and
\be
\label{ia5}
Z_\ve=Z_\ve(\theta):=\frac{\ve^2 F^2}{(d\, \pi)^2}\sin^2\, \left(\frac{d\, \pi\, \theta}{\ve}\right)=O(\ve^2 F^2).
\ee

 Since $X_\ve\ge 1/4$ for small $\ve$, for such $\ve$ we deduce from \eqref{ia2} that
 \be
 \label{ia6}
 h_\ve(\theta)\le (X_\ve)^{1/2}+|Y_\ve|+(Z_\ve)^{1/2}.
 \ee
 
  Integrating \eqref{ia6} over $(0,\ve)$ and using \eqref{eq:48}, \eqref{ia4} and \eqref{ia5} yields
\be
\label{ia7}
\int_0^\ve h_\ve(\theta)\, d\theta\le \int_0^\ve (X_\ve(\theta))^{1/2}\, d\theta+ o(\ve^2).
\ee 

From \eqref{ia3} we have
\be
\label{ia8}
(X_\ve)^{1/2}=1- \frac{\ve\alpha}{d\, \pi}\sin^2\, \left(\frac{d\, \pi\, \theta}{\ve}\right)+ O(\ve^2).
\ee

Combining \eqref{eq:42}, \eqref{ia7} and \eqref{ia8} we obtain
\bes
\int_{\so}|\dot g-\dot f|\le \frac{2|d|\, \pi}{\ve}\left(\ve-\frac{\ve\alpha}{d\, \pi}\int_0^\ve \sin^2\, \left(\frac{d\, \pi\, \theta}{\ve}\right)+o(\ve^2)\right)=2|d|\, \pi -\ve|\alpha|+o(\ve),
\ees
so that \eqref{eq:49} holds for sufficiently small $\ve$.
\end{proof}
% The proof of \rprop{prop:15A} (ii) is an easy consequence of  \rlemma{lem:strict}.
\begin{proof}[Proof of Lemma \ref{ic}]
  It suffices to consider the case where $d_1>0$, so by
assumption $d_2-d_1<0$. Since
$\int_{{\mathbb S}^1}(f\land\dot{f})=2\pi d_1>0$, the set
\begin{equation*}
A:=\{\zeta\in {\mathbb S}^1;\, \zeta\text{ is a Lebesgue point of $\dot{f}$ with
  $(f\land\dot{f})(\zeta)>0$}\}, 
\end{equation*}
 has positive measure.
Applying
\rlemma{lem:strict} to any point $\zeta\in A$ we conclude that there exists $g\in{\cal
  E}_{d_2}$ for which \eqref{eq:49} holds.
\end{proof}
\begin{proof}[Proof of \rprop{prop:attain}]

{} ${}$

\noindent
{\it Step 1.} Proof of item {\it 1}\\
Assume without loss of generality
  that $0<d_1<d_2$. Let $f\in\mathcal{E}_{d_1}$
  satisfy \eqref{eq:wedge}. By \eqref{weak}, $d_{W^{1,1}}(f,\mathcal{E}_{d_2})\geq
  2\pi(d_2-d_1)$. Since $\Dist_{W^{1,1}}({\cal E}_{d_1}, {\cal
    E}_{d_2})=2\pi(d_2-d_1)$ (by \eqref{XX}) we obtain that $f$
  satisfies \eqref{eq:att}.
 On the other hand, for $f\in\mathcal{E}_{d_1}$ for which
 \eqref{eq:wedge} does not hold we conclude from \rlemma{lem:strict}
 that $d_{W^{1,1}}(f,\mathcal{E}_{d_2})<\Dist_{W^{1,1}}({\cal E}_{d_1}, {\cal
    E}_{d_2})=2\pi(d_2-d_1)$, so \eqref{eq:att} does not hold for $f$.
  
 For $f\in\mathcal{E}_{d_1}$ satisfying condition \eqref{eq:stric} (we may take for example $f(\zeta)=\zeta^{d_1}$)
  we get from \eqref{eq:str} that \eqref{eq:strong-att} is violated
  (although \eqref{eq:att} holds). Finally to show that
  \eqref{eq:strong-att} occurs for some $f$, choose
  $\varphi\in W^{1,1}((0,2\pi))$ such that for some $a\in(0,2\pi)$ we have: \\
{\rm (i)} $\varphi'\geq 0$ on $[0,a]$.\\
{\rm (ii)} $\varphi(0)=0, \varphi(a)=2\pi d_1$.\\
{\rm (iii)} $\varphi=2\pi d_1$ on $[a,2\pi]$.\\
Next define $\psi$ on $[0,2\pi]$ by:
\begin{equation*}
  \psi(\theta)=\begin{cases} \varphi(\theta), &\text{ for }\theta\in[0,a]\\
           2\pi\, d_1+2\pi\, (d_2-d_1)\, \d\frac{\theta-a}{2\pi-a},&\text{ for
           }\theta\in(a,2\pi]
\end{cases}.
\end{equation*}
 Setting $f(e^{\im\theta})=e^{\im\varphi(\theta)}$ and
 $g(e^{\im\theta})=e^{\im\psi(\theta)}$ we clearly have $f\in
 \mathcal{E}_{d_1}$ and $g\in \mathcal{E}_{d_2}$. Since $f$ satisfies
 \eqref{eq:wedge} we know
 that $d_{W^{1,1}}(f,\mathcal{E}_{d_2})=2\pi(d_2-d_1)$. But clearly  also
 $|f-g|_{W^{1,1}}=2\pi\, (d_2-d_1)$.

\smallskip
\noindent
{\it Step 2.} Proof of item {\it 2}\\
 The result follows directly from \rlemma{ic} and
\eqref{XX}.
\end{proof}
\begin{remark}
\label{rem:zero}
  If $d_1=0$ and $d_2\neq 0$ then for every {\em non constant} $f\in
  {\cal E}_0$ we have $d_{W^{1,1}}(f,\mathcal{E}_{d_2})<\Dist_{W^{1,1}}({\cal E}_{0}, {\cal
    E}_{d_2})=2\pi|d_2|$. This implies that a constant map is the only
  map for which \eqref{eq:att} holds.
  Indeed, since $\int_{{\mathbb S}^1}(f\land\dot{f})=0$, there are
  Lebesgue points of $f\land\dot{f}$ of both positive and negative
  sign. Hence, for every $d_2\neq0$ we can find a Lebesgue point for
  which  \eqref{eq:67} is satisfied, and the result follows from
  \rlemma{lem:strict}.
\end{remark}

\section{$W^{1, p}$ maps, with $1<p<\infty$}
\label{ub6}

\begin{proof}[Proof of Theorem \ref{bgl20} when $1<p<\infty$]
%As above, we can ask whether,  given $d_1\neq d_2$, the infimum  
%\be
%\label{ub80}
%\inf_{f\in{\cal E}_{d_1}}\ \inf_{g\in{\cal E}_{d_2}}\ d_{W^{1,p}}(f, g)=2^{(1/p)+1}\pi^{(1/p)-1}\, |d_1-d_2|
%\ee
%is achieved. The answer is given by the following result, proved in \cite{rs} when $p=2$.
%\bpr
%\label{ub90}
%Let $d_1, d_2\in\Z$, $d_1\neq d_2$.
%\ben
%\item
%Let $1<p<2$. The infimum in \eqref{ub80} is achieved if and only if $d_2=-d_1$.
%\item
%Let $p\geq 2$.  The infimum in \eqref{ub80} is not achieved.
%\een
%\epr
%\begin{proof}[Proof of Prop~\ref{ub90}]
We  first sketch the proof of the inequality \enquote{$\geq$} in \eqref{ub80}.
  Given any $f\in\mathcal{E}_{d_1}$ and
  $g\in\mathcal{E}_{d_2}$,  set $w:=f\, \overline{g}\in\mathcal{E}_{d}$, with $d:=d_1-d_2$. 
Let $\widetilde w:=T\circ w\in\mathcal{E}_{d}$ where, as in
  \cite{rs,bms},
$T:\so\to\so$ is defined  by
 \begin{equation}
\label{eq:41}
     T(e^{\im \theta})=e^{\im \va}\text{ with }\va=\va(\theta)=\pi\sin(\theta/2),\ \forall\,\theta\in(-\pi,\pi].
   \end{equation}
 Noting that $|e^{\im \theta}-1| =\d\frac{2}{\pi}|\va|$, we obtain as in \cite{rs, bms} (with $\p_\tau$ standing for the tangential derivative)
  
  \be
    \label{eq:1}
    \begin{aligned}
  \int_{\so} |\p_\tau (f-g)|^p&\geq  \int_{\so} \left|\p_\tau |f-g|\right|^p=\int_{\so}
  \left|\p_\tau|f\,\overline {g}-1|\right|^p=
  \int_{\so}\left|\p_\tau |w-1|\right|^p\\&= 
\left(\frac{2}{\pi}\right)^p\int_{\so} |\p_\tau\widetilde w |^p\geq \left(\frac{2}{\pi}\right)^p\inf_{v\in \mathcal{E}_d}\int_{\so} |\dot v|^p.
  \end{aligned}
  \ee
The inequality \enquote{$\geq$} in \eqref{ub80} clearly follows from
\eqref{eq:1} and the next claim:
\begin{equation}
  \label{eq:2}
 \min_{v\in \mathcal{E}_d}\int_{\so} |\dot v|^p=2|d|^p\, \pi.
 \end{equation}
 
 To verify \eqref{eq:2} we first associate to each $v\in
 \mathcal{E}_d$ a function $\psi\in W^{1,p}((-\pi, \pi) ;\R)$ such that
 $v(e^{\im\theta})=e^{\im\psi(\theta)}$, $\theta\in[-\pi,\pi]$, with
 $\psi(\pi)-\psi(-\pi)=2d\, \pi$. We then have, invoking H{\"o}lder inequality,
\begin{equation*}
\int_{\so} |\dot v|^p=\int_{-\pi}^{\pi}|\psi'|^p\geq \frac{(2|d|\, \pi)^p}{(2\pi)^{p-1}},
\end{equation*}
whence the inequality \enquote{$\geq$} in \eqref{eq:2}. On the other hand,
the function $\widetilde w(e^{\im\theta})=e^{\im d\theta}$ clearly gives
equality in \eqref{eq:2}, completing the proof of \eqref{eq:2}. Note
that $\widetilde w$ is the unique minimizer in \eqref{eq:2}, up to rotations.
 The proof of the inequality \enquote{$\leq$} in \eqref{ub80} can be carried out
 using an explicit construction, like the proof in \cite{rs} for
 $p=2$.\end{proof}

Next we turn to the question of attainment of the infimum in
\eqref{ub80}. 
\begin{proof}[Proof of Theorem \ref{ub90}, items 2 and 3]
The proof of the case $p\geq2$ is identical to the one
given in \cite{rs} for $p=2$, so we consider here  only item 
{\it 3} (i.e., we let  $1<p<2$).

\smallskip
\noindent
{\it Step 1.} The infimum in \eqref{ub80} is achieved when $d_2=-d_1$\\
Assume that  $d_2=-d_1$. Let
$d:=d_1-d_2=2d_1$. We saw above that $\widetilde w(e^{\im\theta})=e^{\im
  d\theta}$ realizes the minimum in \eqref{eq:2}. Consider
$S:=T^{-1}:\so\to\so$ (see \eqref{eq:41}), given explicitly by
\bes
S(e^{\im\theta})=e^{\im\psi},  \text{ with }
\psi(\theta)=2\arcsin (\theta/\pi),\,\forall\theta\in [-\pi,\pi].
\ees
 Although $S$ is not Lipschitz, we do have $w:=S\circ \widetilde w\in
 W^{1,p}(\so;\so)$ (i.e., $w\in \mathcal{E}_d$). Indeed, this amounts
 to $\d\frac{1}{\sqrt{1-t^2}}\in L^p((1-\delta,1))$, which holds since
 $p<2$. 
%(For a related argument, see Step 4 in the proof of Lemma \ref{vd1}
%in the Appendix.) 
Since $d$ is even and $w$ has degree $d$, there exists $f\in \mathcal{E}_{d_1}$
 satisfying $w=f^2$. We let
 $g:=\overline f\in\mathcal{E}_{d_2}$, so that $w=f\,\overline{g}$.
 Note that $f-g$ takes only purely imaginary values, and therefore 
 \begin{equation}
\label{eq:4}
   |\p_\tau (f-g)|=\left|\p_\tau|f-g|\right|\  \text{ a.e. on }\so.
 \end{equation}
For these particular $f, g, w$ and $\widetilde w$, we get, using
\eqref{eq:4} that all the inequalities in \eqref{eq:1} are actually
equalities, and we see that the infimum in \eqref{ub80} is attained.

\smallskip
\noindent
{\it Step 2.}
If the infimum in \eqref{ub80} is achieved, then $d_2=-d_1$\\
Assume that the infimum in \eqref{ub80} is achieved by two
functions $f\in\mathcal{E}_{d_1}$ and $g\in {\cal E}_{d_2}$. Set $d:=d_1-d_2$, $w:=f\, \overline{g}$
and $\widetilde w:=T\circ w$. We then have $w,\widetilde w\in\mathcal{E}_{d}$. We may
assume that $d>0$.
From the fact that both inequalities in \eqref{eq:1} must be
equalities we deduce that\\[2mm]
{\rm (i)} $\widetilde w$ is a minimizer in \eqref{eq:2}\\
and\\
{\rm (ii)} \eqref{eq:4} holds.\\[2mm]
From {\rm (i)} it follows that $\widetilde w(e^{\im\theta})=e^{\im
  (d\theta+C)}$ for some constant $C$, and we may assume
that $C=0$. Therefore, 
\begin{equation*}
  w^{-1}(1)=\widetilde{w}^{-1}(1)=\{1,\omega,\omega^2,\ldots,\omega^{d-1}\},\text{
    with }\omega=e^{\im 2\pi/d}.
\end{equation*}
 On the small arc $I_j$ between $\omega^j$ and $\omega^{j+1}$ we may write
 $f-g=\rho\, e^{\im\psi}$ with $\rho=|f-g|$ and $\psi\in W^{1,p}_{loc}$, and we have
 \begin{equation*}
   |\p_\tau (f-g)|^2=\rho^2[\dot\psi]^2+[\dot\rho]^2.
 \end{equation*}
 By  {\rm (ii)}, $\dot\psi=0$ on $I_j$, so that $\psi$ is constant on $I_j$, say $\psi=\alpha_j$ on $I_j$. The equality $f-g=\rho\, e^{\im \alpha_j}$ on $I_j$ implies that $g=e^{\im (2\alpha_j-\pi)}\, \overline f$ on $I_j$, and therefore $g\land \dot g=- f\land \dot f$ on each $I_j$. Since this is true on   each $I_j$, we finally conclude that $d_2=-d_1$.
\end{proof}

\section{$\wsp$ maps, with $sp>1$}
\label{ua7}

%\subsection{Distance in $\wsp$, $sp>1$} 
%\label{bgn1}
%
%Let $0<s<\infty$ and $1\leq p<\infty$ be such that $sp>1$. In this functional setting, we consider the adapted classes
%\bes
%{\cal E}_{d_1}:=\{ f\in W^{s,p}(\so ; \so);\, \deg f=d\}=\{ e^{\im\va(z)}\, z^d;\, \va\in W^{s,p}(\so ; \R)\}
%\ees
%and distance
%\bes
%\dist ({\cal E}_{d_1}, {\cal E}_{d_2}):=\inf\{ |f-g|_{\wsp};\, f\in {\cal E}_{d_1},\ g\in{\cal E}_{d_2}\}.
%\ees
%\bpr
%\label{bgn2}
%Let $0<s<\infty$ and $1\leq p<\infty$ be such that $sp>1$. 
%Then
%\be
%\label{bgn3}
%\dist ({\cal E}_{d_1}, {\cal E}_{d_2})\sim |d_1-d_2|^s,\ \fo d_1, d_2\in\Z.
%\ee
%\epr
\begin{proof}[Proof of Theorem \ref{uc3}, item 2]

${}$

\noindent
{\it Step 1.} 
%Proof of \enquote{$\lesssim$} in \eqref{ue2} when $d_1$ and $d_2$ have opposite signs\\
%Let $f_d(z)=z^d$, $d\in\Z$. We rely on the following calculation \cite[Lemma 2.1]{mironescusickel}: 
%\be
%\label{bgp2}
%|f_d|_{\wsp}\sim |d|^s.
%\ee
%
%Then 
%\bes
%\dist_{\wsp} ({\cal E}_{d_1}, {\cal E}_{d_2})\leq |f_{d_1}-f_{d_2}|_{\wsp}\leq |f_{d_1}|_{\wsp}+|f_{d_2}|_{\wsp}\sim |d_1|^s+|d_2|^s\lesssim |d_1-d_2|^s.
%\ees
%
%\smallskip
%\noindent
%{\it Step 2.} 
Proof of \enquote{$\lesssim$} in \eqref{ue2}\\
%We may assume that $d_1>d_2>0$. Let 
%\bes
%f (e^{\im\theta}):=\begin{cases}
%e^{2d_2\im\theta},&\text{if }0\leq \theta\leq \pi\\
%e^{2(d_1-d_2)\im\theta},&\text{if }\pi<\theta\leq 2\pi
%\end{cases},\ g (e^{\im\theta}):=\begin{cases}
%e^{2d_2\im\theta},&\text{if }0\leq \theta\leq \pi\\
%1,&\text{if }\pi<\theta\leq 2\pi
%\end{cases}.
%\ees 
Fix a smooth map $h\in {\cal E}_1$ such that $h(z)\equiv 1$ when Re $z\le 0$. 

 Given $d_2$, consider a smooth map $g\in {\cal E}_{d_2}$ such that $g(z)\equiv 1$ when Re $z\ge 0$. Set $f:=h^{d_1-d_2}\, g\in {\cal E}_{d_1}$. Then 
%Then $f\in {\cal E}_{d_1}$, $g\in {\cal E}_{d_2}$,  and we claim that 
\be
\label{bgn5}
|f-g|_{\wsp}\lesssim |d_1-d_2|^s. 
\ee

Indeed, estimate \eqref{bgn5} is clear when $s$ is an integer, since $f-g=h^{d_1-d_2}-1$. The general case follows via  Gagliardo-Nirenberg. 

\smallskip
\noindent
{\it Step 2.} Proof of \enquote{$\gtrsim$} in \eqref{ue2} when $0<s\le 1$\\
We rely on an argument similar to the one in   Step 2 in the proof of Theorem \ref{bgl20} in Section \ref{sec:ua5}. Assume that $d:=d_1-d_2>0$, and that $f(1)=g(1)$. Write $f(e^{\im\theta})=e^{\im\va(\theta)}\, g(e^{\im\theta})$, with $\va\in W^{s,p}((0, 2\pi))$ and $\va(0)=0$.  Let  $0= t_0<\tau_0<\cdots< \tau_{d-1}<t_{d}= 2\pi$ be such that $(f-g)(e^{\im t_j})=0$ and $|(f-g)(e^{\im\tau_j})|=2$. By scaling and the hypotheses $0<s\le 1$ and $sp>1$, we have
\be
\label{bgn6}
|h(b)-h(a)|\lesssim (b-a)^{s-1/p}|h|_{\wsp(( a, b))},\ \fo a <  b,\   \fo h\in \wsp((a,b)).
\ee
Applying \eqref{bgn6} to $h:=(f-g)(e^{\im\theta})$ on $(a,b):=(t_j, \tau_j)$, $j=0,\ldots, d-1$, we obtain that $|h|_{\wsp((t_j, \tau_j))}\gtrsim 1/(\tau_j-t_j)^{s-1/p}$, and thus 
\bes
|f-g|_{\wsp}^p\gtrsim \sum_{j=0}^{d-1}|h|_{\wsp((t_j, \tau_j))}^p\gtrsim \sum_{j=0}^{d-1}\frac 1{(\tau_j-t_j)^{sp-1}}\gtrsim d^{sp}, 
\ees
the latter inequality following from Jensen's inequality applied to the function $x\mapsto 1/x^{sp-1}$, $x>0$.

\smallskip
\noindent
{\it Step 3.} Proof of \enquote{$\gtrsim$} in \eqref{ue2} when $s> 1$\\
%In this case we define e.g. by induction on the integer part of $s$
%\bes
%|f|_{\wsp}:=\|\dot f\|_{W^{s-1,p}},\ \text{with }\|h\|_{W^{s-1,p}}:=|h|_{W^{s-1,p}}+\|h\|_{L^p}.
%\ees
The key ingredient in Step 4 is the Gagliardo-Nirenberg type inequality 
\be
\label{zb1}
|f|_{W^{\theta s, p/\theta}}\le C_{\theta, s, p}|f|_{\wsp}^{\theta}\|f\|_{L^\infty}^{1-\theta},\ \fo s>0,\ 1\le p<\infty\text{ such that }(s,p)\neq (1,1),\ \fo \theta\in (0,1).
\ee

Let us note that, if $f, g:\so\to\so$ and $\deg f\neq \deg g$, then (by the argument leading to \eqref{vl1}) 
\be
\label{zb5}
\|f-g\|_{L^\infty}=2.
\ee

By \eqref{zb1} and \eqref{zb5}, we find that for every  $s$, $p$, $\theta$ as in \eqref{zb1} we have
\be
\label{zb2}
\dist_{\wsp}({\cal E}_{d_1}, {\cal E}_{d_2})\ge C'_{\theta, s, p}[\dist_{W^{\theta s, p/\theta}}({\cal E}_{d_1}, {\cal E}_{d_2})]^{1/\theta}, \   \fo d_1, d_2\in\Z.
\ee

If we take, in \eqref{zb2}, $\theta$ such that $\theta s<1$, we obtain Step 4 from Step 3.
 \end{proof}

  \section{Maps $f:\sn\to\sn$}
 \label{ve1}
%   We shall consider the distance computed  using the following 
%     semi-norms or norms:
%   \begin{enumerate}
%   \item In $W^{s,p}(\sn;\R^k)$, $s\in(0,1)$, 
%  \begin{equation*}
%    |f|_{\wsp}^p=\int_{\sn}\int_{\sn}\frac{|f(x)-f(y)|^p}{|x-y|^{sp+N}}\, dxdy.
%  \end{equation*}
%   \item In $W^{1,p}(\sn;\R^k)$, 
%       $|f|_{W^{1,p}}:=\|d f\|_{L^p(\sn)}$.  
% \item In $W^{s,p}(\sn;\R^k)$, $s>1$,  we first consider a full norm obtained via localization and the use of local coordinates. More specifically, let $(\varphi_j,U_j, V_j)$ be a finite atlas of $\sn$ where $\varphi_j: U_j\subset \sn\to V_j\subset \R^N$ is a smooth diffeomorphism. Let $(\theta_j)$ be a partition of unity subordinated to the covering $(U_j).$  We let 
%\bes
%\| f \|_{W^{s,p}}:=\sum_{j}\left\|(\theta_j\, f)\circ(\varphi_j)^{-1}\right\|_{W^{s,p}(V_j)}.
%\ees
%
%Different choices of finite coverings  and/or partitions of unity   lead to the same space
%$W^{s,p}(\sn ; \R^k)$, with equivalent norms. 
%
%Finally, we set 
%\bes
%|f|_{\wsp}:=\left\|f-\fint_{\sn} f  \right\|_{\wsp}
%\ees
%and
%\bes
%d_{\wsp}(f,g):=|f-g|_{\wsp}.
%\ees
% \begin{equation*}
%    |f|_{\wsp}^p=\sum_{|\alpha| =m} |\partial^\alpha f|_{W^{\sigma,p}}^p.%+\|Du\|_{p}^p
%  \end{equation*}
  % \end{enumerate}
   
  \subsection{A useful construction}
  \lab{tA3}

   Throughout Section \ref{ve1} we will make an extensive use of special smooth maps $f:\sn\to\sn$, $N\ge 1$. Such maps \enquote{live} on a small spherical cap, say $B_R(\sigma)$, where $B_R(\sigma)$ is the geodesic ball of radius $R<1$ centered at some point $\sigma$ of $\sn$, and are constant on $\sn\setminus B_R(\sigma)$. Since the construction is localized we may as well work first on a flat ball $B_R(0)$ centered at $0$ in $\R^N$ and then we will transplant $f$ to $B_R(\sigma)$, thereby producing a map $f:\sn\to\sn$. On $B_R(0)$, the map $f$ is determined by a smooth function $F:[0,R]\to\R$ and a smooth map $h:{\mathbb S}^{N-1}\to {\mathbb S}^{N-1}$. 
   
   For simplicity we start with the case $N\ge 2$ since the case $N=1$ is somewhat \enquote{degenerate} and will be discussed later. 
   
    Fix a smooth function $F:[0,R]\to\R$ satisfying
    \begin{gather}
    \lab{B.1}
    F(r)=0\text{ for }r\text{ near }0.\\
    \lab{B.2}
    F(r)=k\, \pi\text{ for }r\text{ near }R,\text{ where }k\in\Z.
    \end{gather}
    
    We may now define $f:B_R(0)\to\sn$ by
    \be
    \lab {B.20}
    f(x)=(\sin F(|x|)\, h(x/|x|),\, (-1)^N\, \cos F(|x|)).
    \ee
    
    Note that $f$ is well defined and smooth on $B_R(0)$ (by \eqref{B.1}) and that $f$ is constant near $\partial B_R(0)$ (by \eqref{B.2}). More precisely 
    \bes
    f(0)=(0,0,\ldots, 0,(-1)^N)=\begin{cases}
  {\n},&\text{if }N\text{ is even}\\
  {\s},&\text{if }N\text{ is odd}
  \end{cases}
  \ees
   and 
  \bes
\text{for }x\text{ near }\partial B_R(0),\   f(x)=(0,0,\ldots, 0, (-1)^N\, \cos k\pi)={\bf C}:=\begin{cases}
  {\n},&\text{if }N+k\text{ is even}\\
  {\s},&\text{if }N+k\text{ is odd}
  \end{cases};
  \ees 
  here ${\n}=(0,0,\ldots, 0,1)$ and ${\s}=(0,0,\ldots 0, -1)$ are the north pole and the south pole of $\sn$. We transport $f$ into $B_R(\sigma)\subset\sn$ via a fixed orientation preserving diffeomorphism and extend it by the value  ${\bf C}$ on $\sn\setminus B_R(\sigma)$.
  %  (with $C=\n$ if $k$ is even and $C=\s$ if $k$ is odd). 
    In this way we have defined a smooth map $f:\sn\to\sn$.

   For the purpose of Lemmas \ref{tA} and \ref{tM} below it suffices to assume that $F:[0,R]\to\R$ is merely continuous and satisfies $F(0)=0$, $F(R)=k\pi$, so that $f:\sn\to\sn$ is a well-defined continuous map.
    
   \bl
   \lab{tA}
  Let $k\in \{ 0, 1\}$. We have 
   \be
   \lab{t1a}
   \deg f=k\, \deg h.
   \ee
   \el
   \begin{proof} We emphasize the fact that here we assume $N\ge 2$, although the conclusion still holds when $N=1$ (see below). 
   %We work on $B_R(0)$ rather than on $\sn$. 

   It will be convenient to assume that $F$ satisfies \eqref{B.1} and \eqref{B.2}; the general case follows by density.
   
   The cases where $k=0$ (respectively $d=0$) are trivial via homotopy to $F\equiv 0$ (respectively $h\equiv C$).    
   With no loss of generality, we assume that $d:=\deg h> 0$ and $k=1$.  
   
   Since $f$ is constant outside $B_R(\sigma)$, it suffices to determine the degree of $f_{|B_R(\sigma)}$, and then we may as well work on the flat ball $B_R(0)\subset\R^N$. We will work in the class of maps 
   \bes
  C^0_{\bf C}(\overline B_R(0) ; \sn):= \{ g:\overline B_R(0)\to \sn;\, g={\bf C}\text{ on }\partial B_R(0)\},
   \ees
   which have a well-defined degree (since they can be identified with maps in $C^0(\sn ; \sn)$). 
   
   \smallskip
   \noindent
   {\it Step 1.} Proof of \eqref{t1a} when $d=1$ and $k=1$\\
   This case can be reduced by homotopy to the case $h=\textrm{Id}$ and $F:[0,R]\to [0,\pi]$ is  non decreasing. 
%   We start with a special case: $h=\textrm{Id}$ and $f:[0,R]\to [0,\pi]$ non decreasing. 
   In this case, for almost every $\bfs\in \sn$ the equation $f(\bft)=\bfs$ has exactly one solution $\bft$, and $f$ is orientation preserving at $\bft$. Thus $\deg f=1$.
   
%   The remaining part of the proof consists of reducing the general case to the above special one, using homotopy arguments.
   
   \smallskip
   \noindent
   {\it Step 2.} Proof of \eqref{t1a} when $d>1$ and $k=1$\\
   Consider smooth maps $h_1, h_2,\ldots, h_d:{\mathbb S}^{N-1}\to {\mathbb S}^{N-1}$ of degree $1$ which \enquote{live} in different regions $\omega_1,\ldots, \omega_d$ of ${\mathbb S}^{N-1}$, in the sense that $\overline{\omega_j}\cap \overline{\omega_k}=\emptyset$ when $j\neq k$ and $h_j=(0,0,\ldots, 0, 1)$ in ${\mathbb S}^{N-1}\setminus \omega_j$, $\fo j$. We glue these maps together and obtain a smooth map $\widetilde h:{\mathbb S}^{N-1}\to{\mathbb S}^{N-1}$ of degree $d$. Since $h$ and $\widetilde h$ are homotopic within $C^\infty ({\mathbb S}^{N-1} ; {\mathbb S}^{N-1})$, the map $f$ and the map $\widetilde f$ corresponding to $\widetilde h$ (via \eqref{B.20}) are homotopic within $C^\infty (B_R(0) ; \sn)$. Thus $\deg f=\deg \widetilde f$.  
   
   On the other hand, let $f_j$ be the map associated to $h_j$ via \eqref{B.20}. Set 
\bes
   \Omega_j:=\{ r\, y ;\,  y\in\omega_j,\, 0<r<R\}.
 \ees

 Note that the $\Omega_j$'s are mutually disjoint.
 %, $1\le j\le d$ and of  $\sn\setminus \cup_j\Omega_j$. 
 
% Then
% \be
% \lab{tJ1}
% \deg f=\sum_j\deg f_{|\Omega_j}+\deg f_{| B_R(0)\setminus \cup_j\Omega_j}.
% \ee
% 
% Note that 
If $x\in \overline B_R(0)\setminus\Omega_j$, then $f_j(x)\in {\cal C}$, where 
 \bes
 {\cal C}:=\{ (0,0,\ldots, 0,\, \sin\theta, \,\cos\theta);\, \theta\in \R\}\subset\sn
 \ees 
(since for such $x$ we have $h(x/|x|)=(0,0,\ldots, 0, 1)$). 
Similarly, if $x\in \overline B_R(0)\setminus\cup_j\Omega_j$, then $f(x)\in {\cal C}$.

 Since ${\cal C}$ has null measure in $\sn$  (here we use $N\ge 2$), we may find some  value $z\in \sn\setminus {\cal C}$  regular for $f$ (and thus for each $f_j$). For such $z$, we have
 \bes
 \deg f=\sum_{x\in f^{-1}(z)}\, \sgn\, \Jac f(x)=\sum_j \sum_{x\in f^{-1}(z)\cap\Omega_j}\, \sgn\, \Jac f(x)=\sum_j\deg f_j=d,
 \ees
 the latter equality following from Step 1.
 \end{proof}

The conclusion of Lemma \ref{tA} also holds for $N=1$ and arbitrary $k$, but this requires a separate argument. When $N=1$, we have ${\mathbb S}^{N-1}={\mathbb S}^0=\{ -1, 1\}$ and we have (modulo symmetry) only two maps $h: {\mathbb S}^0\to {\mathbb S}^0$, namely
\begin{gather*}
h_1(-1)=-1,\ h_1(1)=1,\\
h_2(-1)=1,\ h_2(1)=1.
\end{gather*}

Then $\deg h_1=1$ and $\deg h_2=0$.

The associated maps $f_1, f_2$ defined on $B_R(0)=(-R, R)$ with values in $\so$ are 
\begin{gather*}
f_1(x)=\begin{cases}
(\sin F(x),\, -\cos F(x)),&\text{if }x>0\\
(-\sin F(-x),\, -\cos F(-x)),&\text{if }x<0
\end{cases},\\
f_2(x)=\begin{cases}
(\sin F(x),\, -\cos F(x)),&\text{if }x>0\\
(\sin F(-x),\, -\cos F(-x)),&\text{if }x<0
\end{cases}.
\end{gather*}

Clearly $f_1=e^{\im\va_1}$ and $f_2=e^{\im\va_2}$, where
\begin{gather*}
\va_1(x)=\begin{cases}
-\pi/2+F(x),&\text{if }x>0\\
-\pi/2-F(-x),&\text{if }x<0
\end{cases},\\
\va_2(x)=\begin{cases}
-\pi/2+F(x),&\text{if }x>0\\
-\pi/2+F(-x),&\text{if }x<0
\end{cases}.
\end{gather*}

Thus
\bes
\deg f_1=\frac 1{2\pi}\, (\va_1(R)-\va_1(-R))=\frac{2 F(R)}{2\pi}=k
\ees
and
\bes
\deg f_2=\frac 1{2\pi}\, (\va_2(R)-\va_2(-R))=0.
\ees

 For the record, we call the attention of the reader to the following generalization of Lemma \ref{tA}
 \bl
 \lab{tM}
 For every $k\in\Z$, 
 \be
 \lab{tM2}
 \deg f=\begin{cases}
 k\, \deg h,&\text{if }N\text{ is odd}\\
 \deg h,&\text{if }N\text{ is even and }k\text{ is odd}\\
 0,&\text{if }N\text{ is even and }k\text{ is even}
 \end{cases}.
 \ee
 \el
 \begin{proof}
 Assume e.g. that $k\ge 2$. [The case $k<0$ is handled similarly and is left to the reader.] 
% Let us note that, once $k$ and $h$ are fixed, the degree of $f$ does not depend on $F$. This is obtained by using a homotopy preserving \eqref{B.1} and \eqref{B.2}  between two $F$'s . When $k=1$, the lemma is proved, by  Step 2. 
%When $k\ge 2$,  

As explained in the proof of Lemma \ref{tA}, we may work in the class $C^0_{\bf C}(\overline B_R(0) ; \sn)$.

We may assume via homotopy that $F(r)=k\, \pi\, r/R$.  Set $r_j=j\, R/k$, $j=0,\ldots, k$.   Consider the  functions
   \bes
   F_j(r):=\begin{cases}
   0,&\text{if }r<r_{j-1}\\
   F(r)-(j-1)\, \pi,&\text{if }r_{j-1}\le r< r_j\\
   \pi,&\text{if }r\ge r_j
   \end{cases},\ j=1,\ldots, k.
   \ees 
   
   Consider also the maps $f_j$ corresponding to $F_j$ via \eqref{B.20}. Then $f$ is obtained by gluing the maps $(-1)^{j-1}\, f_j$. By Lemma \ref{tA}, we have
   \be
   \lab{tM4}
   \deg f_j=\deg h,\  j=1,\ldots, k.
   \ee
   
   We next note that 
   \be
   \lab{tM5}
  \text{for every }g\in C^0_{\bf C}(\overline B_R(0) ; \sn),\  \deg (-g)=\begin{cases}
   \deg g,&\text{if } N\text{ is odd}\\
    -\deg g,&\text{if }N\text{ is even}
   \end{cases}.
   \ee
   
  By \eqref{tM4} and \eqref{tM5},  we have
   \bes
   \deg f=\sum_j \deg \left((-1)^{j-1}\, f_j\right)=\begin{cases}
 k\, \deg h,&\text{if }N\text{ is odd}\\
 \deg h,&\text{if }N\text{ is even and }k\text{ is odd}\\
 0,&\text{if }N\text{ is even and }k\text{ is even}
 \end{cases}.\qedhere
   \ees
 \end{proof}

\subsection{Proof of Theorem \ref{vo1}, item {\it 2}}
 
\smallskip
 \noindent
 {\it Step 1.} Proof of the lower bound in  \eqref{ve5}\\
 Since we assume that 
  \be
  \label{ib1}
  [s>0\text{ and }sp>N]\text{ or }[s= N\text{ and }p=1],
  \ee
the space $\wsp$ is embedded continuously in the
space of continuous functions, and there exists a constant $C_{N, s, p}$
such that
\begin{equation}
  \label{eq:3}
\left\|f-\fint_{\sn} f  \right\|_{L^\infty}\leq C_{N, s, p} |f|_{\wsp},\ \fo f\in \wsp.
\end{equation}

 Step 1 is a direct consequence of the
next lemma.
\begin{lemma}
  \label{lem:lb}
  In all spaces $\wsp$ satisfying \eqref{ib1} we have, for all $f\in {\cal E}_{d_1}$,
  $g\in  {\cal E}_{d_2}$, $d_1\neq d_2$,
  \begin{equation}
    \label{eq:5}
d_{\wsp}(f,g)\geq\frac{1}{C_{N, s, p}},
 \end{equation}
 where $C_{N, s, p}$ is the constant in \eqref{eq:3}.
\end{lemma}
\begin{proof}[Proof of Lemma \ref{lem:lb}]
%  We use the same argument as in \cite[Lemma 4]{ls}. We claim that
%  there exists a point $\widetilde x\in\sn$ such that $f(\widetilde
%  x)=-g(\widetilde x)$. Indeed, this is obtained by repeating the argument leading to \eqref{vl1}. The
%  map $I(t,x):[0,1]\times{\sn}\to\R^{N+1}$ given by
%  $I(t,x)=t\,f(x)+(1-t)\,g(x)$ must take the value zero for some
%  $\widetilde x\in\sn$ and $t\in(0,1)$, for otherwise $I(t,x)/|I(t,x)|$ would be
%  a homotopy between $u_1$ and $u_2$. Clearly we must have $t=1/2$,
%  so that $f(\widetilde x)=-g(\widetilde x)$ as claimed. With no loss of generality, we may assume that
%  \begin{equation*}
%    g(\widetilde x)=e_{N+1}=(0,\ldots,0,1).
%  \end{equation*}
% We can also assume that $d_2\neq 0$, hence there exists a point $\widetilde
% y\in\sn$ such that $g(\widetilde y)=-e_{N+1}$. Denoting by $\rho$ the
% $(N+1)$-th coordinate of $g-f$ we have: $\rho(\widetilde x)=2$ and
% $\rho(\widetilde y)\leq 0$. Applying \eqref{eq:3} to $\rho$ yields
% \begin{equation*}
%   2\leq |\rho(\widetilde x)-\rho(\widetilde y)|\leq C_X|\rho|_X\leq C_X|f-g|_X,
% \end{equation*}
%and \eqref{eq:5} follows.
Recall (see \eqref{vl1}) that
\be
\label{ib2}
\|f-g\|_{L^\infty}= 2.
\ee
 
 From \eqref{eq:3} we have
 \be
 \label{ib3}
 \left\| (f-g)-\fint_{\sn} (f-g)  \right\|_{L^\infty}\le C_{N, s, p}|f-g|_{\wsp},
 \ee
 so that 
 \be
 \label{ib4}
 2= \|f-g\|_{L^\infty}\le |A|+r,
 \ee
 where $A:=\fint_{\sn}(f-g)$ and $r:=C_{N, s, p}|f-g|_{\wsp}$.
 
  We may assume that $A\neq 0$, otherwise \eqref{eq:5} is clear. From \eqref{ib3} we have
 \be
 \label{ib5}
 f(\sn)\subset \sn +A+\overline B(0, r).
 \ee
 
 Clearly, 
 \bes
 -\frac A{|A|}\not\in \sn +A+\overline B(0, r)\text{ if }|A|>r,
 \ees
 and then $f$ cannot be surjective -- so that $\deg f=0$. Similarly, we have $\deg g=0$. This is impossible since $d_1\neq d_2$, and  therefore
 \be
 \label{ib6}
 |A|\le r=C_{N, s, p}|f-g|_{\wsp}.
 \ee
 
 Combining \eqref{ib4} and \eqref{ib6} yields $1\le C_{N, s, p}|f-g|_{\wsp}$.
\end{proof}

\noindent
{\it Step 2.} Proof of the upper bound in \eqref{ve5}\\
%  The lower bound following from
%   \rlemma{lem:lb},  we turn to the proof of the upper bound for
%   $\dist_{W^{s,p}}$ in \eqref{ve5}.
%For the convenience of the reader we shall give a detailed proof in
%the case $N=2$ and explain how to adapt the proof for treating general $N$.
 % We will use the following spherical coordinates on the sphere $\sn$
% (see \cite{hua}),
% writing each $\bfs=(s^1,\ldots, s^{N+1})\in\sn$ as 
% \begin{equation}
% \label{eq:coord}
% \begin{aligned}
%   \bfs=(\cos\varphi,\sin\varphi&\cos\theta_1,\sin\varphi\sin\theta_1\cos\theta_2,\ldots,\\
%             &\sin\varphi\sin\theta_1\cdots\sin\theta_{N-2}\cos\theta_{N-1},
%             \sin\varphi\sin\theta_1\cdots\sin\theta_{N-2}\sin\theta_{N-1}),
%   \end{aligned}
% \end{equation}
%  with
%  \begin{equation*}
%    \varphi,\theta_1,\ldots,\theta_{N-2}\in[0,\pi],\, \theta_{N-1}\in[0,2\pi].
%  \end{equation*}
%   For later use it will be convenient to introduce the following notation for
%   functions enjoying ``$m$-axial symmetry''. Namely, for a given integer
%   $m$ and a function $\Phi(\varphi)$ we denote by $F_{m,\Phi}$ the
%   function given by
%   \begin{equation}
%     \label{eq:Fm}
% \begin{aligned}
%   F_{m,\Phi}(\bfs)=(\cos\Phi(\varphi),\sin\Phi(\varphi)&\cos\theta_1,\sin\Phi(\varphi)\sin\theta_1\cos\theta_2,\ldots,\\
%             &\sin\Phi(\varphi)\sin\theta_1\cdots\sin\theta_{N-2}\cos (m\theta_{N-1}),\\
%             &\sin\Phi(\varphi)\sin\theta_1\cdots\sin\theta_{N-2}\sin (m\theta_{N-1})).
%   \end{aligned}
%   \end{equation}
We will construct maps $f\in {\cal E}_{d_1}$, $g\in {\cal E}_{d_2}$, constant outside some small neighborhood $B_R({\n})$  of the north pole ${\n}=(0,0,\ldots, 0, 1)$ of $\sn$, satisfying \eqref{ve5}. We will use the setting described in Section \ref{tA3}. 

%Since the construction is localized, we may as well work in a flat ball  of small (but fixed) radius $\ve$ centered at $0$ in $\R^N$.  

We start with the case $d_1=d$, $d_2=0$. Let $h:{\mathbb S}^{N-1}\to{\mathbb S}^{N-1}$ be any  smooth map of degree $d$. [Here we use the assumption $N\ge 2$. If $N=1$, such an $h$ does not exist when $|d|\ge 2$; see the discussion in Section \ref{tA3} concerning the case $N=1$.]  Let $G:[0,R]\to\R$ be a smooth function such that 
\bes
G(r)=\begin{cases}
0,&\text{if }r\le R/4\\
\pi/2,&\text{if }R/3\le r\le 2R/3\\
0,&\text{if }3R/4\le r\le R
\end{cases}.
\ees

Let $F:[0,R]\to\R$ be defined by
\bes
F(r):=\begin{cases}
G(r),&\text{if }0\le r< R/2\\
\pi- G(r),&\text{if }R/2\le r\le R
\end{cases}.
\ees

Clearly, $F$ and $G$ satisfy assumptions \eqref{B.1} and \eqref{B.2}. 

We now define as in Section \ref{tA3}
\begin{gather*}
f(x)=(\sin F(|x|)\, h(x/|x|),\, (-)^N\,\cos F(|x|)),\\
g(x)=(\sin G(|x|)\, h(x/|x|),\, (-1)^N\, \cos G(|x|)).
\end{gather*}

From Lemma \ref{tA} we have $\deg f=d$ and $\deg g=0$. Clearly 
\bes
\sin F(r)=\sin G(r),\ \fo r\in [0,R],
\ees
and thus
\bes
f(x)-g(x)=\begin{cases}
0,&\text{if }|x|<R/2\\
(0,0,\ldots, 0, 2\, (-1)^N\cos F(|x|)),&\text{if }R/2\le |x|<R
\end{cases}.
\ees

% and $A\in C^\infty ([0,\ve); [0,1])$ be such that $A(r)=\begin{cases}
%0,&\text{if }r\le \ve/4\\
%1,&\text{if }\ve/3\le r\le 2\ve/3\\
%0,&\text{if }r\ge 3\ve/4
%\end{cases}
%$. 
%
%Define
%\bes
%g(x)=\left( \sin \frac\pi 2 A(|x|)\, h\left(\frac x {|x|}\right),\,  \cos \frac\pi 2 A(|x|)\right) ,\ \fo x\in\R^N\text{ with }|x|<\ve,
%\ees
%and
%\bes
%f(x)=\begin{cases}g(x),&\text{if }|x|<\ve/2\\
%\d\left( \sin \frac\pi 2 A(|x|)\, h\left(\frac x {|x|}\right),\,  -\cos \frac\pi 2 A(|x|)\right) ,&\text{if }\ve/2<|x|<\ve \end{cases}.
%\ees
%
%
%Let us note that $f$ and $g$ are smooth, and that $f\in {\cal E}_d$ and $g\in {\cal E}_0$. [This is seen using the definition of the degree as $\sum_{x\in f^{-1}(a)}\Jac f(x)$, where $a$ is a regular value of $f$.] 

In the case where $d_1=d$ and $d_2=0$, the upper bound \eqref{ve5} follows from the fact that $f-g$ does not depend on $d$.

We next turn to the general case. Consider a map $m\in C^\infty (\R^N; \sn)$ such that $m(x)=\text{\n}$ when $|x|>R/4$ and $\deg m=d_2$. Then, with $d:=d_1-d_2$ and with $f$ and $g$  as above, consider 
\bes
\widetilde f(x)=\begin{cases}
m(x),&\text{if }|x|<R/4\\
f(x),&\text{if }R/4\le |x|<R
\end{cases},\ \ \widetilde g(x)=\begin{cases}
m(x),&\text{if }|x|<R/4\\
g(x),&\text{if }R/4\le |x|<R
\end{cases}.
\ees

Then $\widetilde f\in{\cal E}_{d_1}$, $\widetilde g\in{\cal E}_{d_2}$, and $\widetilde f-\widetilde g=f-g$, whence \eqref{ve5}.
\hfill$\square$

\subsection{Proof of Theorem \ref{vo1}, item {\it 1}}
\lab{tA7}

Here $N\ge 1$. A key ingredient  is the following
 \bl
 \label{yX}
 There are two families of smooth maps $f_\ve, g_\ve : \sn\to\sn$, defined for $\ve$ small, such that
 \begin{gather}
 \label{yX1}
 f_\ve(\bfs)=g_\ve(\bfs)={\n},\ \fo \bfs\in B_{\ve/4}({\s}),\\
 \label{yX2}
 f_\ve(\bfs)={\s},\ \fo \bfs\in\sn\setminus B_{\ve^{1/2}}({\s}),\\
  \label{yX3}
 g_\ve(\bfs)={\n},\ \fo \bfs\in\sn\setminus B_{\ve^{1/2}}({\s}),\\
 \label{yX4}
 \deg f_\ve=1,\\
 \label{yX5}
 \deg g_\ve=0,\\
 \label{yX6}
 |f_\ve-g_\ve|_{W^{N/p, p}(\sn)}\to 0\text{ as }\ve\to 0,\ \fo 1<p<\infty.
 \end{gather}
 \el
% The key ingredient is the following
% \bl
% \label{id1}
%Let $N\ge 1$. There exists a family $F^\ve:\sn\to [0,1]$ of smooth functions such that 
% \ben
% \item
% $F^\ve(\bfs)$ depends only on $s^1$ (where $\bfs=(s^1,\ldots, s^{N+1})$).
% \item
% $F^\ve(\bfs)=0$ if $s^1\ge 1-\ve$.
% \item
% $F^\ve(\bfs)=1$ if $s^1\le 1-\ve^{1/2}$.
% \item
% $F^\ve$ is non increasing in $s^1$.
% \item
% For $1<p<\infty$, $\|F^\ve-1\|_{W^{N/p, p}}\to 0$ as $\ve\to 0$.
% \een
% \el
% 
% Lemma \ref{id1} is reminiscent of the fact that in the critical spaces $W^{N/p, p}$, $1<p<\infty$, the capacity of a point is zero; its proof is postponed to the Appendix. 
% \br
% The lemma also holds for $N=1$. This is obtained  using arguments similar to the one in Step 4 in the proof of Lemma \ref{vd1}, but is not needed here.
% \er

 Granted Lemma \ref{yX} we proceed with the 
  \begin{proof}[Proof of \rth{vo1}, item 1] Assume e.g. that $d:=d_1-d_2>0$. We fix $d$ distinct  points ${\sigma_1},\ldots, \sigma_d\in\sn$.  Note that $f_\ve-{\s}$ has support in $B_{\ve^{1/2}}({\s})$. Therefore, for sufficiently small $\ve$, we may glue $d$ copies of $f_\ve$ centered at ${\sigma_1},\ldots, \sigma_d\in\sn$. We denote by $\widetilde f_\ve$ the resulting map. By construction $\widetilde f_\ve-\s$ is supported in the union of mutually disjoint balls $B_{\ve^{1/2}}(\sigma_i)$, $i=1,\ldots, d$. From \eqref{yX4} we have
  \be
  \label{yX7}
  \deg \widetilde f_\ve=d.
  \ee

 Next we consider a family of smooth maps $h_\ve:\sn\to\sn$ such that
 \be
 \label{yX8}
 \deg h_\ve=d_2
 \ee
 and
 \be
 \label{yX9}
 h_\ve({\bfs})={\n},\ \fo \bfs\in\sn\setminus B_{\ve/8}(\sigma_1).
 \ee
 [The construction of $h_\ve$ is totally standard.]
 
 We glue $h_\ve$ to $\widetilde f_\ve$ by inserting it in $B_{\ve/8}(\sigma_1)$ (here we use \eqref{yX1}). The resulting map is denoted by $\widehat f_\ve$. From \eqref{yX7} and \eqref{yX8} we have 
 \be
 \label{yX10}
 \deg\widehat f_\ve=d+d_2=d_1,
 \ee
 so that $\widehat f_\ve\in {\cal E}_{d_1}$. 
 
 We proceed similarly with $g_\ve$ using the same points $\sigma_1,\ldots, \sigma_d\in\sn$. We first obtain $\widetilde g_\ve$ such that, by \eqref{yX5}, 
 \be
 \label{yX11}
 \deg\widetilde g_\ve=0.
 \ee
 
 We then glue $h_\ve$ to $g_\ve$ as above and obtain some $\widehat g_\ve$ such that, by \eqref{yX8} and \eqref{yX11}, 
 \be
 \label{yX12}
 \deg\widehat g_\ve=0+d_2=d_2,
 \ee
 so that $\widehat g_\ve\in{\cal E}_{d_2}$. 
 
 Clearly $\widehat f_\ve-\widehat g_\ve$ consists of $d$ glued copies of $f_\ve-g_\ve$. Therefore 
 \bes
 \left|\widehat f_\ve-\widehat g_\ve \right|_{W^{N/p, p}}\le d\, |f_\ve-g_\ve|_{W^{N/p, p}}
 \ees
 and thus
 \bes
\dist_{W^{N/p,p}} ({\cal E}_{d_1}, {\cal E}_{d_2})\le \left|\widehat f_\ve-\widehat g_\ve \right|_{W^{N/p, p}}\to 0\text{ as }\ve\to 0.
 \qedhere
 \ees
 \end{proof}
 
 We now turn to the 
 \begin{proof}[Proof of Lemma \ref{yX}]
 Since the construction is localized on a small geodesic ball, we may as well work on the flat ball $B_R(0)$ centered at $0$ in $\R^N$, with $R>{\ve^{1/2}}$. 
 
 Fix a smooth nonincreasing function $K: \R\to [0,1]$ such that 
 \be
 \label{z1}
 K(t)=\begin{cases}
 1,&\text{if }t\le 1/4\\
 0,&\text{if }t\ge 3/4
 \end{cases}.
 \ee
 
 Consider the family of radial functions
 $H_\ve(x)=H_\ve(|x|):\R^N\to [0,1]$ defined by 
 \be
 \label{z2}
 H_\ve(x)=H_\ve(|x|):=\begin{cases}
 \d K\left(\frac 14-\frac 1{2\ln 2}\, \ln\left(\frac{\ln 1/|x|}{\ln 1/\ve}\right)\right),&\text{if }|x|<1\\
 0,&\text{if }|x|\ge 1
 \end{cases}.
 \ee
 
 Here, $\ve$ is a parameter such that
 \be
\label{z3}
0<\ve<1/e^2.
\ee

We also consider the radial functions 
  $F_\ve(r)$ and $G_\ve(r)$ defined by 
 \be
 \lab{tG1}
 F_\ve(r):=\begin{cases}
 \d \pi\,  (1-K(r/\ve))/2,&\text{if }r<\ve\\
 \d \pi\,  (1- H_\ve(r)/2),&\text{if }\ve\le r<R
 \end{cases}
 \ee
 and
 \be
 \lab{tG2}
 G_\ve(r):=\begin{cases}
 F_\ve(r),&\text{if }r<\ve\\
 \d\pi- F_\ve(r)= \pi\, H_\ve(r)/2,&\text{if }\ve\le r<R
 \end{cases}.
 \ee
 
% \begin{figure}[h]
%\begin{center}
%\includegraphics[scale=0.3]{fig1.pdf}
%\caption{The shape of $F_\ve$}
%\end{center}
%\end{figure}

%\begin{figure}
%\centering
%\begin{minipage}{.5\textwidth}
%  \centering
%  \includegraphics[width=.4\linewidth]{image1}
%  \captionof{figure}{A figure}
%  \label{fig:test1}
%\end{minipage}%
%\begin{minipage}{.5\textwidth}
%  \centering
%  \includegraphics[width=.4\linewidth]{image1}
%  \captionof{figure}{Another figure}
%  \label{fig:test2}
%\end{minipage}
%\end{figure}

 Note that $F_\ve$ and $G_\ve$ are smooth (this is clear in the regions $\{r<\ve\}$ and $\{ r>3\ve/4\}$).
 
 \begin{figure}[h]
\centering
\begin{subfigure}{.5\textwidth}
  \centering
  \includegraphics[width=0.9\linewidth]{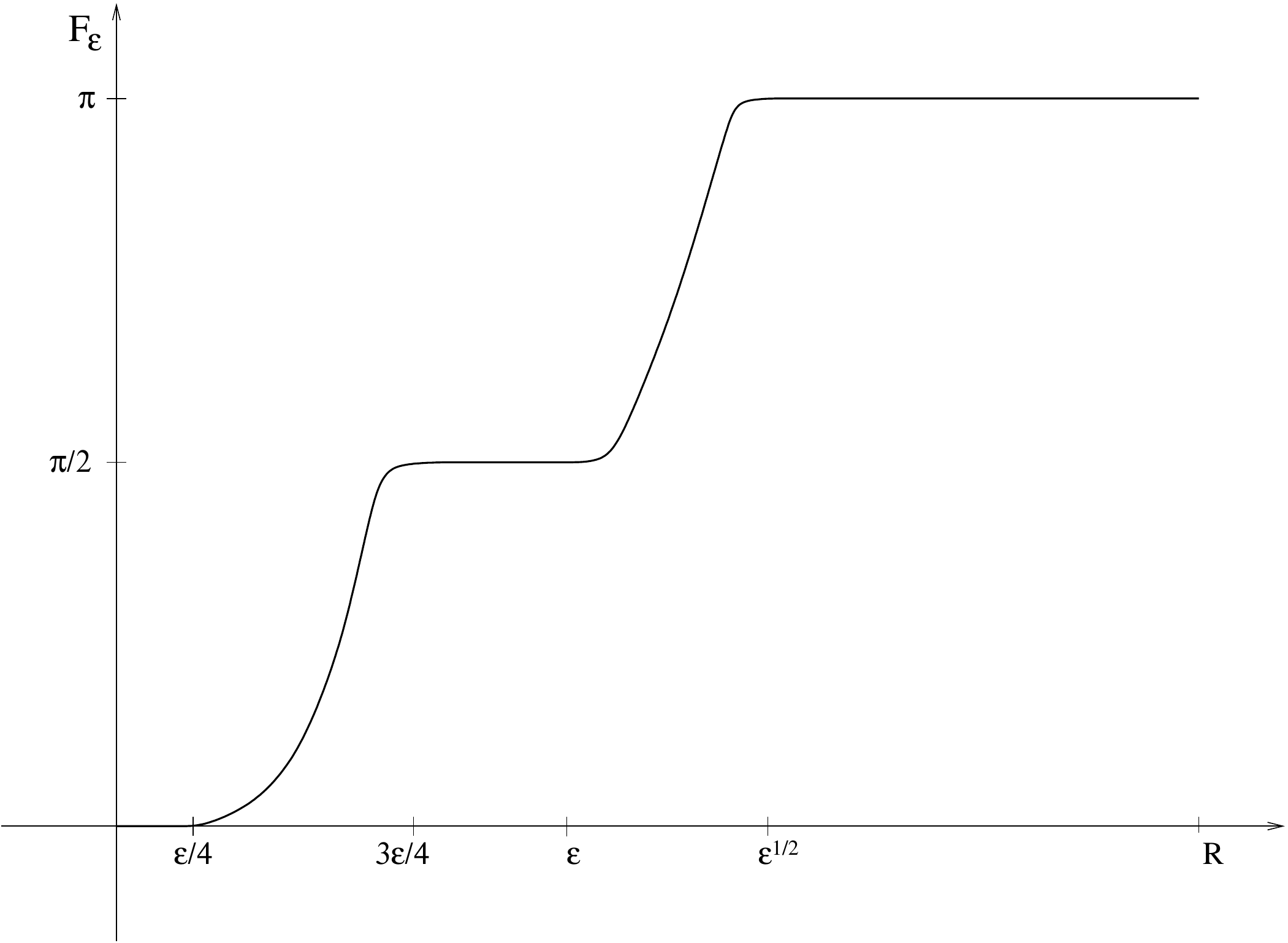}
  %\caption{Shape of $F_\ve$}
  \label{fig:sub1}
\end{subfigure}%
\begin{subfigure}{.5\textwidth}
  \centering
  \includegraphics[width=0.9\linewidth]{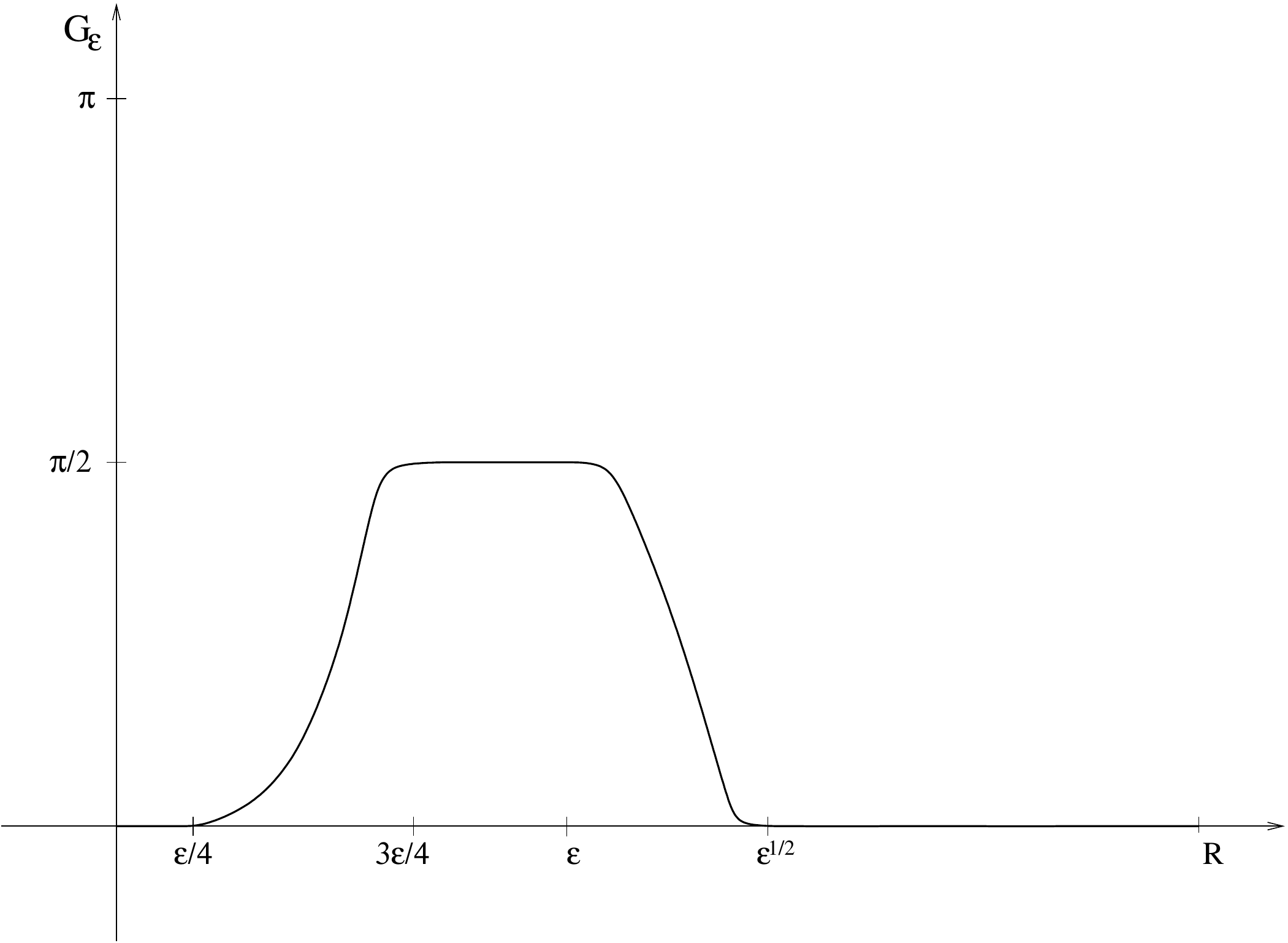}
  %\caption{Shape of $G_\ve$}
  \label{fig:sub2}
\end{subfigure}
\caption{Plots of $F_\ve$ and $G_\ve$ given by \eqref{tG1} and \eqref{tG2}}
\label{fig:test}
\end{figure}

 As in Section \ref{tA3} set
 \begin{gather*}
 f_\ve(x)=\left(\sin F_\ve(|x|)\, \frac x{|x|},\, (-1)^N\,\cos F_\ve(|x|)\right),\ \fo x\in B_R(0),\\
  g_\ve(x)=\left(\sin G_\ve(|x|)\, \frac x{|x|},\, (-1)^N\,\cos G_\ve(|x|)\right),\ \fo x\in B_R(0).
 \end{gather*}

% Using \eqref{z1}--\eqref{z3}, and also \eqref{A.5}--\eqref{A.7} in   Lemma \ref{AA} in the Appendix, we see that $G_\ve$ satisfies
% \begin{gather}
% \lab{yX14}
% G_\ve(r)=0\ \text{if }r<\ve/4\\
% \lab{yX15}
% G_\ve(r)\text{ is non decreasing on }(\ve/4, 3\ve/4)\\
% \lab{yX16}
% G_\ve(r)=H_\ve(r)=1\text{ on }(3\ve/4, \ve)\\
% \lab{yX17}
% G_\ve(r)\text{ is non increasing on }(\ve, \infty)\\
% \lab{yX18}
% G_\ve(r)=0\text{ if }r\ge \ve^{1/2}.
% \end{gather}
% 
% Note that $G_\ve$ is smooth on $\R^N$ (this is clear in the regions $\{ |x|<\ve\}$ and $\{ |x|>3\ve/4\}$). 
% 
% Set 
% \be
% \lab{yX19}
% A_\ve(r)=\sin \left(\frac \pi 2\, G_\ve(r)\right),\ \fo r>0,
% \ee
% \be
% \lab{yX20}
% B_\ve(r)=\cos \left(\frac \pi 2\, G_\ve(r)\right),\ \fo r>0,
% \ee
% 
%and
%\be
%\lab{yX21}
%C_\ve(r)=\begin{cases}
%B_\ve(r),&\text{if }r\le\ve\\
%-B_\ve(r),&\text{if }r>\ve
%\end{cases}.
%\ee
%
%Clearly $A_\ve(x)=A_\ve(|x|)$ and $B_\ve(x)=B_\ve(|x|)$ are smooth on $\R^N$. This is also true for $C_\ve$ (consider separately the regions $\{ |x|<\ve\}$ and $\{ |x|>3\ve/4\}$; in the latter, we have $C_\ve(|x|):=-\cos \d\frac \pi 2\, G_\ve(|x|)$). 

%Finally we define the maps $f_\ve, g_\ve :\R^N\to\sn$ by
%\begin{gather}
%\lab{yX23}
%f_\ve(x)=\left( A_\ve (|x|)\, \frac x{|x|},\, C_\ve (|x|) \right),\\
%\lab{yX22}
%g_\ve(x)=\left( A_\ve (|x|)\, \frac x{|x|},\, B_\ve (|x|) \right).
%\end{gather}

It is clear (using Lemma \ref{tA}) that \eqref{yX1}--\eqref{yX5} hold. Moreover,
\bes
f_\ve(x)-g_\ve(x)=\left(0,0,\ldots, 0, 2\, (-1)^{N+1}\, \cos \left(\frac \pi 2 H_\ve(|x|)\right)\right),\ \fo x\in B_R(0),
\ees
(since $H_\ve(r)=1$ when $r<\ve$ by \eqref{z2}).
%Clearly $f_\ve$ and $g_\ve$ are smooth. We deduce \eqref{yX1}, \eqref{yX2} and \eqref{yX3} from \eqref{yX14} and \eqref{yX18}--\eqref{yX21}. Note that $\deg g_\ve=0$ since $g_\ve$ takes its values in $\sn_+$ by \eqref{yX22}. On the other hand $f_\ve$ covers the full sphere $\sn$ \enquote{monotonically} as $r$ increases from $0$ to $\ve^{1/2}$; thus $\deg f_\ve=1$.
%
%Finally we have by \eqref{yX20}--\eqref{yX23} and \eqref{yX13} 
%\bes
%f_\ve(x)-g_\ve(x)=\begin{cases}
%(0,0,\ldots, 0, 0),&\text{if }|x|\le \ve\\
%\d (0,0,\ldots, 0, -2 B_\ve(|x|),&\text{if }|x|>\ve
%\end{cases},
%\ees
%so that in fact
%\bes
%f_\ve(x)-g_\ve(x)=\left(0,0,\ldots, 0, -2 \cos \left(\frac\pi 2\, H_\ve (|x|)\right)\right),\ \fo x\in\R^N
%\ees
%(since $H_\ve(r)=1$ when $r\le \ve$, by \eqref{A.5}).
%
Therefore
\bes
|f_\ve- g_\ve|_{W^{N/p,p}}=2\, \left| \cos \left(\frac \pi 2\, H_\ve\right) \right|_{W^{N/p, p}}.
\ees

Consider the function
\bes
\widetilde K(r)=1-\cos \left(\frac \pi 2\, K(r)\right),\ \fo r\in\R.
\ees

Clearly $\widetilde K$ satisfies \eqref{z1}. Consider the function $\widetilde H_\ve$ derived from $\widetilde K$ via \eqref{z2}, so that
\bes
\widetilde H_\ve(x)=1-\cos \left(\frac \pi 2\, H_\ve (x)\right),\ \fo x\in\R^N,
\ees
and therefore
\bes
|f_\ve-g_\ve|_{W^{N/p, p}(\R^N)}=2\left|\widetilde H_\ve \right|_{W^{N/p, p}(\R^N)}\to 0\text{ as }\ve\to 0
\ees
by \eqref{A.8} in Lemma \ref{AA} (applied to $\widetilde K$).
  \end{proof}

\subsection{Proof of Theorem \ref{vp1}, item 1 (and of Theorem \ref{vm1}, item 2)}

We rely on the following result, whose proof is postponed to  the Appendix.
\bl
\label{ie1}
Let $N\ge 1$ and $1\le p<\infty$. Fix a geodesic ball $B\subset\sn$ (of small radius). Then there exists a map $h:\sn\to\sn$ (depending on $d$) such that 
\ben
\item
$\deg h=d$.
\item
$h=(0,0,\ldots, 0,1)$ outside $B$.
\item
$|h|_{W^{N/p,p}}\le C_{N,p}|d|^{1/p}$.
\een
\el

Granted Lemma \ref{ie1}, we proceed as follows. Let $g\in {\cal E}_{d_2}$ be a smooth map such that $g$ is constant in a neighborhood of some closed ball $B$. Such maps are dense in ${\cal E}_{d_2}$, and with no loss of generality we assume that $g=(0,0,\ldots, 0,1)$ near $B$. Let $h$ be as in the above lemma, with $d:=d_1-d_2$, and set $f=\begin{cases}
g,&\text{in }\sn\setminus B\\
h,&\text{in }B
\end{cases}$. Then clearly $f\in {\cal E}_{d_1}$ and 
\be
\label{ie2}
\dist_{W^{N/p, p}}(g, {\cal E}_{d_1})\le |f-g|_{W^{N/p, p}}\le C_{N,p}|d_1-d_2|^{1/p}.
\ee

The validity of \eqref{ie2} for arbitrary $g\in {\cal E}_{d_2}$ follows by density.\hfill$\square$

\subsection{Proof of Theorem \ref{vp1}, item {\it 2} (and of Theorem \ref{vm1}, item {\it 3})}

This time the key construction is provided by the following 
\bl
\label{if1}
Let $N\ge 1$. Fix $d_1\in\Z$. Then there exists a sequence of smooth maps $f_n:\sn\to\sn$ (with sufficiently large $n$) such that:
\ben
\item
$\deg f_n=d_1$.
\item
For every geodesic ball $B\subset\sn$ of radius $1/n$, $f_n(B)=\sn$.
\een
\el

Granted Lemma \ref{if1}, we claim that the sequence $(f_n)$ satisfies
\be
\label{if2}
\dist_{\wsp}(f_n, {\cal E}_{d_2})\ge C'_{s, p, N,\alpha}n^\alpha,\ \text{with }C'_{s, p, N,\alpha}>0,
\ee
for any $0<\alpha\le 1$ such that $\wsp\hookrightarrow C^\alpha$. Clearly, the desired result follows from \eqref{if2}.

In order to prove \eqref{if2}, we argue by contradiction. Then, possibly along a subsequence still denoted $f_n$, there exist maps $g_n\in {\cal E}_{d_2}$ such that 
\be
\label{if3}
|f_n-g_n|_{C^\alpha}=o(n^\alpha)\text{ as }n\to\infty;
\ee
here, we consider the $C^\alpha$ semi-norm 
\bes
|f|_{C^{\alpha}}:=\sup\left\{ \frac{|f(x)-f(y)|}{|x-y|^\alpha};\, x, y\in\sn,\ x\neq y\right\}.
\ees
By \eqref{ib2}, for each $n$ there exists a point $\bfs=\bfs_n$ such that $g_n(\bfs)=-f_n(\bfs)$. With no loss of generality, we may assume that $f_n(\bfs)=(0,\ldots, 0,1)$ and therefore $g_n(\bfs)=(0,\ldots,0,-1)$. Let $h_n$ denote the last component of $f_n-g_n$ and let $B_n$ denote the ball of radius $1/n$ centered at $\bfs$. By \eqref{if3}, we have $h_n\ge 2-o(1)$ in $B_n$. On the other hand, Lemma \ref{if1}, item {\it 2}, implies that there exists some ${\bf t}\in B_n$ such that $f_n({\bf t})=(0,\ldots,0,-1)$. It follows that $h_n({\bf t})\le 0$. This leads to a contradiction for large $n$, and thus \eqref{if2} is proved.
\hfill$\square$

\section{Some partial results towards Open Problems \ref{sa3}, \ref{sa3'} and \ref{tD2}}
\label{sa4}

 \subsection{Full answer to Open Problem \ref{sa3'} when $N=1$ or $2$, $1\le p\le 2$, and $d_1\, d_2\ge 0$}
\lab{tB1}
 
 We start with the special cases [$N=1$, $p=2$] and [$N=2$, $p=2$]. In this cases, we are able to  determine the exact value of $\Dist_{\wsp}({\cal E}_{d_1}, {\cal E}_{d_2})$ provided $d_2>d_1\ge 0$ (Propositions \ref{sf1a}, \ref{sf1b} and their consequences in Proposition \ref{sf1c}). This allows us to  give a positive answer to Open Problem \ref{sa3'} when $N=2$ and $1\le p\le 2$ under the extra assumption that $d_1\, d_2\ge 0$ (Corollary \ref{sf1}).
 
 \bpr
 \lab{sf1a}
 Assume that $N=1$ and $d_2>d_1\ge 0$. Let $f(z)=z^{d_1}$, $z\in\so$. Then 
\be
\lab{tI1}
|f-g|_{H^{1/2}}^2\ge 4\pi^2\, (d_2-d_1),\ \fo g\in {\cal E}_{d_2}.
\ee
 \epr
 \begin{proof}
 We will use the Fourier decomposition of $g\in H^{1/2}(\so ; \so)$, given by $g(e^{\im\theta})=\sum_{n=-\infty}^\infty a_n\, e^{\im\, n\theta}$. Recall (see e.g. \cite{B1}) 
 that the Gagliardo semi-norm \eqref{uc8} has a simple form
 \be
 \label{sg1}
 |g|_{H^{1/2}}^2=4\pi^2\, \sum_{n=-\infty}^\infty |n|\, |a_n|^2
 \ee
 and that for every $g\in H^{1/2}(\so ; \so)$,
 \be
 \label{sg2}
 \deg g=\sum_{n=-\infty}^\infty n\, |a_n|^2,
 \ee
 \be
 \label{sg3}
 \sum_{n=-\infty}^\infty |a_n|^2=1.
 \ee
 
% Without loss of generality we may assume that $d_2>d_1\ge 0$. Set $f(e^{\im\theta}):= e^{\im d_1\theta}\in {\cal E}_{d_1}$. We will prove that 
% \be
% \label{sg4}
% |f-g|_{H^{1/2}}^2\ge 4\pi^2\, (d_2-d_1),\ \fo g\in {\cal E}_{d_2},
% \ee
% which implies 
% that
% \be
% \lab{tH1}
% \Dist_{H^{1/2}}({\cal E}_{d_1},{\cal E}_{d_2})\ge 4\pi^2\, |d_1-d_2|,
% \ee
% and in particular
% \eqref{sf2} for $N=1$, $p=2$. 

By \eqref{sg1} we have
 \bes
 %\label{gb1}
 \begin{aligned}
 \frac 1{4\pi^2}\, |f-g|_{H^{1/2}}^2&= \sum_{\substack{n\in\Z \\ n\neq d_1}} |n|\, |a_n|^2 +d_1\, |a_{d_1}-1|^2=\sum_{n\in\Z} |n|\, |a_n|^2 +d_1\, (|a_{d_1}-1|^2-|a_{d_1}|^2)\\
 &=\sum_{n\in\Z} |n|\, |a_n|^2 +d_1\, (1- 2\,\text{Re}\, a_{d_1})\ge d_2-d_1,
 \end{aligned}
 \ees
 by \eqref{sg2} and \eqref{sg3}.
 \end{proof}
 \bpr
 \lab{sf1b}
 Assume that $N=2$ and $d_2>d_1\ge 0$. Let $f\in {\cal E}_{d_1}$ be defined by $f(\bfs)= {\cal T}^{-1}\Big(\big({\cal
  T}({\bf s})\big)^{d_1}\Big)$ where ${\cal T}:\st\to\C$ is the stereographic
  projection. Then
  \be
  \lab{tI2}
  |f-g|_{H^1}^2\ge 8\pi\, (d_2-d_1),\ \fo g\in {\cal E}_{d_2}.
  \ee
 \epr
 \begin{proof}
 Recall that $f$ is a harmonic map and that 
  \be
  \label{sg5}
  \int_{\st} |\na f|^2=8\pi\, d_1;
  \ee
  see e.g. \cite{B2} and the references therein. For any $g\in {\cal E}_{d_2}$, write 
  \bes
  %\label{sg6}
  \begin{aligned}
  |f-g|_{H^1}^2&=\int_{\st}|\na (f-g)|^2=\int_{\st}|\nabla f|^2-2\int_{\st}
    |\nabla g|^2(g\cdot f)+\int_{\st}|\nabla g|^2\\
    &\geq\int_{\st}|\nabla g|^2-\int_{\st}|\nabla f|^2=\int_{\st}|\nabla g|^2-8\pi\, d_1\geq 8\pi\, ( d_2-d_1), 
 \end{aligned}
  \ees
 by \eqref{sg5} and Kronecker's formula \eqref{wa1}. %Therefore, we have
% \be
% \lab{tH3}
% \Dist_{H^1}({\cal E}_{d_1}, {\cal E}_{d_2})\ge 8\pi\, |d_1-d_2|,
% \ee
% and in particular \eqref{sf2} holds. 
 \end{proof}
 \bpr
 \lab{sf1c}
 Let $d_1, d_2\in\Z$ be such that $d_2>d_1\ge 0$. 
 \ben
 \item
 When $N=1$ we have
 \be
 \lab{tI3}
 \Dist_{H^{1/2}}({\cal E}_{d_1}, {\cal E}_{d_2})=(4\pi^2\, |d_1-d_2|)^{1/2}.
 \ee
 \item 
 When $N=2$ we have
 \be
 \lab{tI4}
 \Dist_{H^{1}}({\cal E}_{d_1}, {\cal E}_{d_2})=(8\pi\, |d_1-d_2|)^{1/2}.
 \ee
 \een 
 \epr
 \begin{proof}
 Formula \eqref{tI4} follows from \eqref{wa2} and \eqref{tI2}.
 
 On the other hand, \eqref{tI3} is a consequence of \eqref{tI1} and of the following one dimensional version of \eqref{wa2}: 
 \be
 \label{gc1}
 \text{Given }\ve>0\text{ and }f\in {\cal E}_{d_1}\text{there exists some }g\in {\cal E}_{d_2}\text{such that }|f-g|_{H^{1/2}}^2\le 4\pi^2\, |d_1-d_2|+\ve.
 \ee
 
 Indeed, let $0<\delta<1$ and set $h_\delta(z):=\d\left(\frac{z-(1-\delta)}{(1-\delta)\, z-1}\right)^{-d}$, with $d:=d_1-d_2$.  Then  $h_\delta\in {\cal E}_{-d}$, and thus $g_\delta:=f\, h_\delta\in {\cal E}_{d_2}$. On the other hand, we clearly have  $h_\delta\to 1$ a.e. as $\delta\to 0$.
 We claim that 
 \be
 \label{gg1}
 |g_\delta-f|_{H^{1/2}}^2=|h_\delta|_{H^{1/2}}^2+o(1)\text{ as }\delta\to 0.
 \ee
 
 Indeed, we start from the identity
 \bes
 (g_\delta -f)(x)-(g_\delta -f)(y)=(h_\delta (x)-1)\, (f(x)-f(y))+(h_\delta(x)-h_\delta(y))\, f(y),
 \ees
 which leads to  the inequalities
 \be
 \label{gg2}
 |(g_\delta-f)(x)- (g_\delta-f)(y)|\ge |h_\delta(x)-h_\delta (y)|-|h_\delta(x)-1|\, |f(x)-f(y)|
 \ee
 and
 \be
 \label{gg3}
 |(g_\delta-f)(x)- (g_\delta-f)(y)|\le |h_\delta(x)-h_\delta (y)|+|h_\delta(x)-1|\, |f(x)-f(y)|.
 \ee
 
 By dominated convergence, we have
 \be
 \label{gg4}
 \int_{\so}\int_{\so}\frac {|h_\delta(x)-1|^2\, |f(x)-f(y)|^2}{|x-y|^2}=o(1)\text{ as }\delta\to 0.
 \ee
 
 Formula \eqref{gg1} is a consequence of \eqref{gg2}--\eqref{gg4}.
 
 Finally, \eqref{gc1} follows from \eqref{gg1} and the fact that  $|h_\delta|^2_{H^{1/2}}=4\pi^2\, |d|$ \cite[Corollary 3.2]{bmrs}.  
 \end{proof}
 \bc
 \label{sf1}
 Assume that $N=1$ or $2$, $1\le p\le 2$ and $d_1\, d_2\ge 0$. Then 
 \be
 \label{sf2}
 H-\dist_{W^{N/p, p}}({\cal E}_{d_1}, {\cal E}_{d_2})\ge C'_{p, N}\, |d_1-d_2|^{1/p}
 \ee
 for some constant $C'_{p, N}>0$.
 \ec
 
 \begin{proof}
% {} ${}$
% 
% \noindent
%  Let $f\in {\cal E}_{d_1}$ be defined by $g(\bfs)= {\cal T}^{-1}\Big(\big({\cal
%  T}({\bf s})\big)^{d_1}\Big)$ where ${\cal T}:\st\to\C$ is the stereographic
%  projection.   
% \smallskip
% \noindent
% {\it Step 3.} Proof of \eqref{sf2} in the remaining cases\\

We may assume that $d_2>d_1\ge 0$, and under this assumption we will prove that  
\be
 \label{sf2*}
 \Dist_{W^{N/p, p}}({\cal E}_{d_1}, {\cal E}_{d_2})\ge C'_{p, N}\, |d_1-d_2|^{1/p}.
 \ee
 
 The case $N=1$, $p=1$ follows from Theorem \ref{vm1}, item {\it 1}. 
 
 The case where $N=1$, $1<p<2$ follows from \eqref{tI1} and the trivial inequality
 \bes
 |f|_{H^{1/2}}^2\le |f|_{W^{1/p, p}}^p\, (2\|f\|_{L^\infty})^{2-p},\ \fo 1<p<2,\ \fo f.
 \ees

 The case where $N=2$ and $1\le p<2$ follows from \eqref{tI2} and the Gagliardo-Nirenberg inequality
 \bes
 |f|_{H^1}^2\le C_{p, N}\, |f|_{W^{2/p,p}}^p\, \|f\|_{L^\infty}^{2-p},\ \fo f.\qedhere
 \ees
 \end{proof}

 \subsection{Full answer to Open Problem \ref{sa3} when $1\le p\le N+1$ and $d_1\, d_2\le 0$}
 \lab{tB2}
 
 In this section we prove that the answer to Open Problem \ref{sa3} is positive when $N\ge 1$, $1\le p\le N+1$ and $d_1\, d_2\le 0$ (Proposition \ref{tN1}). This implies that the answer to Open Problem \ref{sa3'} is positive when $N=1$ or $2$  and $1\le p\le 2$ (Corollary \ref{ha8}).  We end with a review of some simple cases of special interest which are still open (see Remark \ref{tB5}). 
 
 \bpr
 \lab{tN1}
Let $N\ge 1$ and  $1\le p\le N+1$. Let  $d_1, d_2\in\Z$ be such that $d_1\, d_2\le 0$. We have 
 \be
 \lab{tN2}
 \Dist_{W^{N/p, p}}({\cal E}_{d_1}, {\cal E}_{d_2})\ge C'_{p, N}\, |d_1-d_2|^{1/p}.
 \ee
 \epr
  
  \begin{proof}

We rely on the following estimate, valid when $1\le p\le N+1$ :
\be
 \lab{tN3}
 |\deg f-\deg g|\le C_{p, N} |f-g|_{W^{N/p,p}}^{p/(N+1)}\left (|f|_{W^{N/p,p}}^{N p/(N+1)}+|g|_{W^{N/p,p}}^{N p/(N+1)}\right),\ \fo f,\, g\in W^{N/p,p}(\sn ; \sn),
 \ee
(see Proposition \ref{tD50} below). 

Fix a canonical $f_1\in {\cal E}_{d_1}$ (for example $f_1(z)=z^{d_1}$ when $N=1$ or the map given by Lemma \ref{ie1} for $N\ge 1$).

This $f_1$ satisfies
\be
\label{tN6}
|f_1|_{W^{N/p, p}}\le C_{p, N}\, |d_1|^{1/p}.
\ee

Therefore, with different constants $C_{p,N}$ depending on $p$ and $N$, but not on $d_1$ or $d_2$, we have

\be
\lab{tN5}
\begin{aligned}
|d_1-d_2|\le & C_{p,N}\, |f_1-g|_{W^{N/p,p}}^{p/(N+1)}\, \left( |d_1|^{N/(N+1)}+|g|_{W^{N/p, p}}^{Np/(N+1)}\right)\\
\le & C_{p, N}\, |f_1-g|_{W^{N/p,p}}^{p/(N+1)}\, \left( |d_1|^{N/(N+1)}+|f_1|_{W^{N/p, p}}^{Np/(N+1)}+|f_1-g|_{W^{N/p, p}}^{Np/(N+1)}\right)\\
\le & C_{p,N}\, |f_1-g|_{W^{N/p,p}}^{p/(N+1)}\, \left( |d_1|^{N/(N+1)}+|f_1-g|_{W^{N/p, p}}^{Np/(N+1)}\right), \ \fo g\in {\cal E}_{d_2}.
\end{aligned}
\ee

Using \eqref{tN5} and the fact that $|d_1|\le |d_1-d_2|$ (since $d_1\, d_2\le 0$), we find that
\bes
|f_1-g|_{W^{N/p,p}}\ge C'_{p, N}\, |d_1-d_2|^{1/p},\ \fo g\in {\cal E}_{d_2},
\ees
whence \eqref{tN2}.
  \end{proof}

Corollary \ref{sf1} and Proposition \ref{tN1} lead to the following
 \bc
 \label{ha8}
 Assume that $N=1$ or $2$ and $1\le p\le 2$.  Then 
 \bes
 H-\dist_{W^{1/p, p}}({\cal E}_{d_1}, {\cal E}_{d_2})\ge C'_{p}\, |d_1-d_2|^{1/p},\ \fo d_1, d_2\in\Z,
 \ees
 for some constant $C'_{p}>0$.
 \ec
 
% We end this section with further evidence in support of Open Problem \ref{sa3}.
% \bpr
% \label{prop:ua1}
% Assume $N=1$ and $1\le p\le 2$. If $d_1\, d_2<0$, then
% \be
% \label{ua2}
% \Dist_{W^{1/p, p}}({\cal E}_{d_1}, {\cal E}_{d_2})\ge C'_p |d_1-d_2|^{1/p}=C'_p\, (|d_1|+|d_2|)^{1/p}.
% \ee
% \epr
% \begin{proof}
% Set $f(z):=z^{d_1}\in {\cal E}_{d_1}$. As above, it suffices to prove that
% \be
% \label{ua3p}
% |f-g|_{H^{1/2}}\ge C'\,  (|d_1|+|d_2|),\ \fo d_2\in\Z\text{ with }d_1\, d_2<0,\, \fo g\in {\cal E}_{d_2}.
% \ee
% 
% With no loss of generality, we may assume that $d_1>0>d_2$. By Remark \ref{tK2}, \eqref{ua3p} holds if $d_1\ge  |d_2|/2$. We may thus assume that $d_1<  |d_2|/2$. If we write $g(e^{\im\theta})=\sum_{n=-\infty}^\infty a_n\, e^{\im\theta}$, then 
% \be
% \label{ua4}
% |z^{d_1}-g|_{H^{1/2}}^2\ge 4\pi^2\, \sum_{\substack{n\in\Z \\ n\neq d_1}} |n|\, |a_n|^2.
% \ee
% 
% On the other hand, by \eqref{sg2}, \eqref{sg3}.
% \be
% \label{ua5}
% |d_2|\le\sum_{n=-\infty}^\infty |n|\, |a_n|^2 \le \sum_{\substack{n\in\Z \\ n\neq d_1}} |n|\, |a_n|^2 +d_1.
% \ee
% 
% By \eqref{ua4}, \eqref{ua5} and the assumption   $d_1<|d_2|/2$, we find that 
% \bes
% |z^{d_1}-g|_{H^{1/2}}^2\ge  4\pi^2\, (|d_2|-d_1)\ge \frac{16 \pi^2}3\, (|d_1|+d_2).\qedhere
% \ees
% \end{proof}
%

\br
\lab{tB5}
We mention here a few cases of special interest not covered by the results in Section \ref{tB1} and \ref{tB2}.
\ben
\item
In view of  Propositions \ref{sf1c}, item {\it 1},  and  Proposition  \ref{tN1}, we know that when $N=1$ and $p=2$ we have
\be
\lab{tC1}
\Dist_{H^{1/2}}({\cal E}_{d_1}, {\cal E}_{d_2})\ge C'\, |d_1-d_2|^{1/2},\text{ if either }0\le d_1<d_2 \text{ or }d_1\, d_2< 0.
\ee

We do not know whether \eqref{tC1} holds in the  case where $0<d_2<d_1$.
\item
Let $N=2$ and $p=2$. We do not know whether the inequality
\be
\label{tC2}
\Dist_{H^1}({\cal E}_{d_1}, {\cal E}_{d_2})\ge C'\, |d_1-d_2|^{1/2}
\ee
(valid  when $0\le d_1<d_2$ or $d_2\, d_1< 0$ by Proposition \ref{sf1c}, item {\it 2}, and Proposition \ref{tN1}),  still holds in the remaining  cases. A more precise question is whether  \eqref{tC2} holds with  $C'=(8\pi)^{1/2}$. 
\een
\er

\subsection{A very partial answer in the general case}
\lab{tB3}

 \bpr
 \label{sb1}
 Let $N\ge 1$ and $1\le p<\infty$. Then for every $d_1\in\Z$ there exists some $C'_{p, d_1}$ such that 
 \be
 \label{sb2}
 \Dist_{W^{N/p,p}}({\cal E}_{d_1}, {\cal E}_{d_2})\ge C'_{p, d_1}\, |d_1-d_2|^{1/p},\ \fo d_2\in\Z.
 \ee
 \epr
 \begin{proof}
 {} ${}$
 
 \smallskip
 \noindent
  {\it Step 1.} Proof of \eqref{sb2} when $d_1=0$\\
 Since any constant map belongs to ${\cal E}_0$ it suffices to show that
 \be
 \label{sb3}
 \inf_{g\in{\cal E}_{d_2}}|g|_{W^{N/p,p}}\ge C'_p|d_2|^{1/p},\ \fo d_2\in\Z.
 \ee
 
 When $p>N$ we rely on \cite[Theorem 0.6]{lddjr}. The case $p=N$ follows from Kronecker's formula \eqref{wa1}, which leads to
 \be
 \label{sb4}
 C'_N\, |d_2|^{1/N}\le |g|_{W^{1,N}},\ \fo g\in {\cal E}_{d_2}.
 \ee 
 
 The case $1\le p<N$ is a consequence of \eqref{sb4} and of the Gagliardo-Nirenberg inequality 
\bes
|g|_{W^{1,N}}\le C\, |g|_{W^{N/p, p}}^{p/N}\, \|g\|_{L^\infty}^{1-p/N}=C\, |g|_{W^{N/p, p}}^{p/N},\ \fo g\in W^{N/p, p}(\sn ; \sn).
\ees

\smallskip
\noindent
{\it Step 2.} Proof of \eqref{sb2} when $d_1\neq 0$\\
As in the proof of Proposition \ref{tN1}, we fix a canonical $f_1\in {\cal E}_{d_1}$ satisfying \eqref{tN6}.
%\be
%\label{sc1}
%|f_1|_{W^{N/p, p}}\le C_p\, |d_1|^{1/p}.
%\ee
%

Next we claim that for every $d_2\in\Z$, $d_2\neq d_1$,
\be
\label{sc2}
\inf_{g\in {\cal E}_{d_2}}\, |f_1-g|_{W^{N/p, p}}=\alpha (d_1, d_2)>0.
\ee

Indeed, we know from Theorem \ref{ih3} that 
\be
\lab{tF1}
\inf_{g\in {\cal E}_{d_2}}\, |f_1-g|_{W^{N/p, p}}=\alpha( f_1, d_2)>0.
\ee

But since $f_1$ is a canonical map in ${\cal E}_{d_1}$ we obtain \eqref{sc2}.

Write, with $g\in {\cal E}_{d_2}$,
\be
\label{sc3}
|f_1-g|_{W^{N/p, p}}\ge |g|_{W^{N/p, p}}-|f_1|_{W^{N/p,p}}\ge C'_p\, |d_2|^{1/p}-C_p\, |d_1|^{1/p},
\ee
by \eqref{sb3} and \eqref{tN6}. Clearly
\be
\label{sc4}
C'_p\, |d_2|^{1/p}-C_p\, |d_1|^{1/p}\ge \frac 12\, C'_p\, |d_2-d_1|^{1/p}
\ee
provided $|d_2|$ is sufficiently large, say $|d_2|\ge C (p, d_1)$. Finally we apply \eqref{sc2} for all values of $d_2$, $|d_2|<C(p, d_1)$, $d_2\neq d_1$, and we obtain
\be
\label{sc5}
\inf_{g\in {\cal E}_{d_2}}\, |f_1-g|_{W^{N/p, p}}\ge D_{p, d_1}\, |d_2-d_1|^{1/p}
\ee
with $D_{p, d_1}>0$, for every $d_2\in\Z$, $|d_2|< C(p, d_1)$. Combining \eqref{sc3}--\eqref{sc5} yields
\bes
\inf_{g\in {\cal E}_{d_2}}\, |f_1-g|_{W^{N/p, p}}\ge C'_{p, d_1}\, |d_1-d_2|^{1/p},\ \fo d_2\in\Z,
\ees
with $C'_{p, d_1}:=\min\{ (1/2)\, C'_p, D_{p, d_1} \}>0$.
\end{proof}

%----------------------------------
%
%Alternative writing for Step 2. 
%
%Fix a map $f_1\in {\cal E}_{d_1}$. For every $d_2\in\Z\setminus\{ d_1\}$ we have
%\be
%\label{sd1}
%\Dist_{W^{N/p,p}}({\cal E}_{d_1}, {\cal E}_{d_2})\ge \inf_{g\in {\cal E}_{d_2}}\, |f_1-g|_{W^{N/p, p}}:=\alpha(p, f_1, d_2)>0,
%\ee
%the latter inequality following from Theorem \ref{ih3}. 
%
%On the other hand, inequality \eqref{sb3} leads to
%\be
%\label{sd2}
%\begin{aligned}
%\Dist_{W^{N/p,p}}({\cal E}_{d_1}, {\cal E}_{d_2})&\ge \inf_{g\in {\cal E}_{d_2}}|f_1-g|_{W^{N/p, p}} \ge \inf_{g\in {\cal E}_{d_2}}(|g|_{W^{N/p, p}}-|f_1|_{W^{N/p, p}})\\
%&\ge C'_{p}\, |d_2|^{1/p}-|f_1|_{W^{N/p,p}}.
%\end{aligned}
%\ee
%
%By \eqref{sd1}--\eqref{sd2} we have
%\bes
%\Dist_{W^{N/p,p}}({\cal E}_{d_1}, {\cal E}_{d_2})\ge \max\{  \alpha (p, f_1,  d_2), C'_{p}\, |d_2|^{1/p}-|f_1|_{W^{N/p,p}}\}\ge C'_{p, d_1}\, |d_1-d_2|^{1/p}, \fo d_2\in\Z\setminus\{ d_1\},
%\ees
%for some sufficiently small $C'_{p, d_1}>0$.
%
%-------------------------------

\subsection{A partial solution to Open Problem \ref{tD2}}
\lab{tD4}
 \bpr
 \lab{tD50}
 Assume that $N\ge 1$ and $1\le p\le N+1$. Then 
 \be
 \lab{tD5}
 |\deg f-\deg g|\le C_{p, N} |f-g|_{W^{N/p,p}}^{p/(N+1)}\left (|f|_{W^{N/p,p}}^{N p/(N+1)}+|g|_{W^{N/p,p}}^{N p/(N+1)}\right),\ \fo f,\, g\in W^{N/p,p}(\sn ; \sn).
 \ee
 \epr
 
 Note that Proposition \ref{tD50} provides a positive answer to Open Problem \ref{tD2} when $N\ge 1$ and $1\le p\le N+1$.
 
 \begin{proof}
 Assuming the case $p=N+1$ proved, the other cases follow via Gagliardo-Nirenberg, with the exception of the case $N=1$, $p=1$. However, in that special case estimate \eqref{tD5} follows from  Theorem \ref{bgl20}. 
 We may thus assume that $p=N+1$. 
 
 Let $F$, $G$ denote respectively the harmonic extension of $f$, $g$ to the unit ball $B$ of $\R^{N+1}$. Then $F, G\in W^{1,N+1}(B ; \R^{N+1})$ and (see e.g. \cite{lddjr})
 \be
 \lab{tE1}
 \deg f=\fint_B \Jac F,\ \deg g=\fint_B \Jac G.
 \ee
 
 Since for any square matrices $A$, $B$ of size $N+1$ we have
 \be
 \lab{tE2}
 |\det A-\det B|\le C\sum_{j=1}^{N+1}\|\text{col}^j(A)-\text{col}^j(B)\|\, \left(\|A\|^{N}+\|B\|^N\right),
 \ee
 we find from \eqref{tE1} and \eqref{tE2} that 
 \be
 \lab{tE3}
 |\deg f-\deg g|\le C\, |F-G|_{W^{1,N+1}}\, (|F|_{W^{1,N+1}}^N+|G|_{W^{1,N+1}}^N).
 \ee
 Finally, we obtain  \eqref{tD5}  from \eqref{tE3} and the estimates
 \bes
 |F|_{W^{1,N+1}}\le C\, |f|_{W^{N/(N+1), N+1}}\text{ and }|G|_{W^{1,N+1}}\le C\, |g|_{W^{N/(N+1), N+1}}.\qedhere
 \ees
 \end{proof}
 
% \br
% \lab{tF2}
% Whenever the answer to Open Problem \ref{tD2} is positive, we may use it in place of \eqref{tF1} and thereby improving the conclusion of Proposition \ref{sb1}.
% \er

 \section*{Appendix. Proofs of some auxiliary results}
 \setcounter{equation}{0}
 \renewcommand{\theequation}{\Alph{section}.\arabic{equation}}
 \setcounter{section}{1}
 \setcounter{proposition}{0}
 \renewcommand{\thelemma}{A}

 \label{vga}
 \addcontentsline{toc}{section}{Appendix. Proofs of some auxiliary results}
 
Let $K:\R\to [0,1]$ be a smooth non increasing function such that 
%Let $K$, $H_\ve$ and $\ve$ be as in \eqref{z1}--\eqref{z3}, i.e., $K$ is non increasing and
 \bes
 K(t)=\begin{cases}
 1,&\text{if }t\le 1/4\\
 0,&\text{if }t\ge 3/4
 \end{cases}.
 \ees
 
 Consider
  the family of radial functions $H_\ve(x)=H_\ve(|x|):\R^N\to [0,1]$, $N\ge 1$, defined by 
 \bes
 H_\ve(x)=H_\ve(|x|):=\begin{cases}
 \d K\left(\frac 14-\frac 1{2\ln 2}\, \ln\left(\frac{\ln 1/|x|}{\ln 1/\ve}\right)\right),&\text{if }|x|<1\\
 0,&\text{if }|x|\ge 1
 \end{cases},
 \ees
 and  $\ve$ is a parameter such that
$
0<\ve<1/e^2$.

The following lemma collects some useful properties of $H_\ve$.

\begin{appl}
\lab{AA}
The functions $H_\ve$ satisfy
\begin{gather}
\lab{A.4}
H_\ve\text{ is smooth on }\R^N,\ \fo\ve.\\
\lab{A.5}
H_\ve(r)=1,\ \fo 0\le r\le\ve,\ \fo \ve.\\
\lab{A.6}
H_\ve(r)=0,\ \fo r\ge \ve^{1/2},\ \fo \ve.\\
\lab{A.7}
H_\ve(r)\text{ is non increasing on }(0,\infty).\\
\lab{A.8}
\text{for every }1<p<\infty,\ \ \|H_\ve(x)\|_{W^{N/p, p}(\R^N)}\to 0\text{ as }\ve\to 0.\\
\lab{A.9}
\text{for every }1<p<\infty\text{ and every }j=1, 2,\ldots, N,\  \|x_j\, H_\ve(x)\|_{W^{1+N/p, p}(\R^N)}\to 0\text{ as }\ve\to 0.
\end{gather}
\end{appl}

Lemma \ref{AA} implies in particular that the $\wsp$-capacity of a point in $\R^N$ is zero when $sp\le N$ and $1<p<\infty$. The above construction is inspired by some standard techniques related to capacity estimates. 
 
 \begin{proof}
 Properties \eqref{A.5}--\eqref{A.7} are obvious. The smoothness of $H_\ve$ is clear (from its definition) in the region $\{|x|<1\}$. It is even clearer from \eqref{A.6} in the region $\{ |x|>\ve^{1/2}\}$ and thus $H_\ve$ is smooth on $\R^N$ since $\ve^{1/2}<1$. 
 
 Consider the function $f:\R^N\to [0,\infty]$ defined by
 \be
 \lab{A.10}
 f(x)=\begin{cases}
 \ln (\ln 1/|x|),&\text{if }|x|<1/e\\
 0,&\text{if }|x|\ge 1/e
 \end{cases}.
 \ee
 
 We claim that 
 \be
 \lab{A.11}
 H_\ve(x)=K(\alpha\, f(x)+\beta_\ve),\ \fo x\in\R^N,
 \ee
 where
 \bes
 \alpha=-\frac 1{2\, \ln 2}\text{ and }\beta_\ve=\frac 14+\frac 1{2\, \ln 2}\, \ln (\ln 1/\ve).
 \ees
 
 Indeed, \eqref{A.11} is clear when $|x|<1/e$. In the region $|x|\ge 1/e$ we have $H_\ve(x)=0$ by \eqref{A.6} (since $1/e\ge \ve^{1/2}$); on the other hand for such $x$ we have $K(\alpha f(x)+\beta_\ve)=0$ since $\beta_\ve\ge 3/4$ (again thanks to the property $1/e\ge \ve^{1/2}$). 
 
 For the proofs of \eqref{A.8} and \eqref{A.9} it is convenient to distinguish the cases $N=1$ and $N\ge 2$.
 
 \smallskip
 \noindent
 {\it Case 1: $N=1$}. We must show that
 \be
 \lab{A.12}
 |H_\ve (x)|_{W^{1/p, p}}(\R)\to 0\text{ as }\ve\to 0
 \ee
 and
 \be
 \lab{A.13}
 |x\, H_\ve(x)|_{W^{1+1/p, p}(\R)}\to 0\text{ as }\ve\to 0.
 \ee
 
 We claim that 
 \be
 \lab{A.14}
 f\in W^{1/p, p}(\R),\ \fo 1<p<\infty.
 \ee
 
 Clearly, it suffices to establish that
 \be
 \label{aad2}
 \iint_{0<y<x<e^{-1}}\frac{|f(x)-f(y)|^p}{(x-y)^2}\, dxdy<\infty,\ \fo 1<p<\infty.
 \ee
 
 With the change of variables $x=e^{-s}$, $y=e^{-s-t}$, $s>1$, $t>0$, inequality \eqref{aad2} amounts to
\be
\label{aad3}
\int_0^\infty\int_1^\infty\frac{[\ln (1+t/s)]^p}{(e^{-s}-e^{-s-t})^2}\, e^{-2s-t}\, dtds=\int_0^\infty\int_1^\infty\frac{[\ln (1+t/s)]^p}{(e^{t/2}-e^{-t/2})^2}\, dtds<\infty.
\ee

In order to prove \eqref{aad3}, we invoke the inequality $\ln (1+t/s)\leq t/s$ and the convergence of the integrals $\d\int_0^\infty \frac{t^p}{(e^{t/2}-e^{-t/2})^2}\, dt$, respectively $\d\int_1^\infty \frac 1{s^p}\, ds$.

Next, we deduce from \eqref{A.11} that 
\be
\lab{A.17}
\frac{|H_\ve(x)-H_\ve(y)|^p}{|x-y|^2}\le C\, \frac{|f(x)-f(y)|^p}{|x-y|^2},\ \fo x, y\in\R.
\ee

Dominated convergence, \eqref{A.17} and \eqref{A.6} imply that 
\bes
|H_\ve|_{W^{1/p, p}(\R)}=\int_\R\int_\R \frac{|H_\ve(x)-H_\ve(y)|^p}{|x-y|^2}\, dxdy\to 0\text{ as }\ve\to 0.
\ees

In view of \eqref{A.12}, property \eqref{A.13} amounts to
\be
\lab{A.18}
|x\, H_\ve'(x)|_{W^{1/p, p}(\R)}\to 0\text{ as }\ve\to 0.
\ee

Clearly
\be
\lab{A.19}
x\, H_\ve'(x)=|\alpha|\, \frac {K'(\alpha f(x)+\beta_\ve)}{\ln 1/|x|},\ \fo x\in\R,
\ee
and thus
\be
\lab{X.20}
x\, H_\ve'(x)=|\alpha|\, \frac {K'(\alpha f(x)+\beta_\ve)}{e^{f(x)}},\ \fo x\in\R
\ee
(note that $x\, H_\ve'(x)=0$ in the region $|x|\ge 1/e$, while $f(x)=\ln (\ln 1/|x|)$ in the region $|x|<1/e$). 

Hence we may write 
\be
\lab{A.21}
x\, H_\ve'(x)=Q_\ve (\alpha f(x)+\beta_\ve),\ \fo x\in\R,
\ee
where
\be
\lab{A.22}
Q_\ve(t)=|\alpha|\, \frac{K'(t)}{e^{(t-\beta_\ve)/\alpha}}=\frac C{\ln 1/\ve}\, \frac{K'(t)}{e^{t/\alpha}},\ \fo t\in\R,
\ee
and $C$ is a universal constant. Clearly $K'(t)\, e^{-t/\alpha}$ belongs to $C^\infty_c(\R)$ and thus is Lipschitz. We deduce from \eqref{A.14},\eqref{A.21} and \eqref{A.22} that
\bes
|x\, H_\ve'(x)|_{W^{1/p, p}(\R)}\le \frac C{\ln 1/\ve}\, |f|_{W^{1/p, p}(\R)}\to 0\text{ as }\ve\to 0.
\ees

\smallskip
\noindent
{\it Case 2: $N\ge 2$}. We must show that for every $1<p<\infty$,
\be
\lab{A.23}
\| H_\ve (x)\|_{W^{N/p, p}(\R^N)}\to 0\text{ as }\ve\to 0
\ee
and
\be
\lab{A.24}
\|x_j\, \nabla H_\ve(x)\|_{W^{N/p, p}(\R^N)}\to 0\text{ as }\ve\to 0.
\ee

We claim that 
\be
\lab{A.25}
\|H_\ve\|_{W^{1, N}(\R^N)}\le \frac C{(\ln 1/\ve)^{(N-1)/N}}\to 0\text{ as }\ve\to 0
\ee
and
\be
\lab{A.26}
\|H_\ve\|_{W^{N,1}(\R^N)}\le C\text{ as }\ve\to 0.
\ee

Assertion \eqref{A.23} with $p>N$ (respectively $p<N$) follows from Gagliardo-Nirenberg, \eqref{A.25} and $\|H_\ve\|_{L^\infty}=1$ (respectively Gagliardo-Nirenberg, \eqref{A.25} and \eqref{A.26}).

For the verification of \eqref{A.25} and \eqref{A.26} note that 
\be
\lab{A.27}
|\partial^\gamma H_\ve (x)|\le \frac{C_k}{\ln 1/\ve}\, \frac 1{|x|^k}\, \one_{M_\ve}(x),\ \fo x\in\R^N,
\ee
for every multi-index $\gamma$ of length $k:=|\gamma|\ge 1$, where
\bes
M_\ve:=\{ x\in \R^N;\, \ve<|x|<\ve^{1/2}\}.
\ees

Assertion \eqref{A.24} is proved in a similar manner using the fact that 
\bes
\|x_j\, \nabla H_\ve (x)\|_{L^\infty(\R^N)}\le \frac C{\ln 1/\ve}.\qedhere
\ees
 \end{proof}

\begin{proof}[Proof of Lemma \ref{ie1}]
We may as well work in a ball $B$ in $\R^N$. We may assume $d>0$. Fix $d$ points $P_1,\ldots, P_d$ in $B$. Consider a smooth map $T:\R^N\to\sn$ such that $T(x)=(1,0,\ldots, 0)$ when $|x|\ge 1$ and $\deg T=1$. For large $n$, let 
\bes
h(x)=\begin{cases} T(n(x-P_j)),&\text{if }|x-P_j|<1/n\text{ for some }j\\
(1,0,\ldots, 0),&\text{otherwise}
\end{cases}.
\ees

Clearly, $h$ satisfies properties {\it 1} and {\it 2}. We claim that $h$ also satisfies {\it 3}. Indeed, this is clear for $p=1$ (by scaling). When $N\ge 2$, the general case follows from Gagliardo-Nirenberg.

When $N=1$, item {\it 3} still holds, but not the above argument, since we do not have $W^{1,1}\hookrightarrow W^{1/p, p}$ when $1<p<\infty$. In order to establish item {\it 3} in $W^{1/p,p}$ with $1<p<\infty$, we fix a small $\delta>0$. Consider the intervals $I_1,\ldots, I_d$ of length $\delta$ centered at $P_1,\ldots, P_d$ and set $I_{d+1}:=B\setminus (I_1\cup\cdots \cup I_d)$. By  straightforward calculations, we have, as $n\to\infty$:
\be
\label{9a1}
\int_{I_j}\int_{I_k}\frac{|h(x)-h(y)|^p}{|y-x|^{1+(1/p)\, p}}\, dxdy=\begin{cases}
C_p+o(1),&\text{if }1\le j=k\le d\\
o(1),&\text{otherwise}
\end{cases};
\ee
this implies that $|h|_{W^{1/p, p}}^p=C_p\, d+o(1)$ and completes the proof of the lemma when $N=1$.
%\begin{proof}[Proof of Lemma \ref{vi2}] It will be more convenient to work with maps defined on $[-\pi, \pi]$ rather than on $\so$. Fix some smooth map $u:\R\to\so$, of degree 1 and such that $u(\theta)=1$ for $|\theta|\geq 1$. Define the rescaled maps $u_\delta(\theta)=u(\theta/\delta)$, $-\pi\leq\theta\leq\pi$. Then clearly $|u_\delta|_{W^{1/p, p}} =|u|_{W^{1/p, p}}+o(1)$ as $\delta\to 0$. Let $h:=\prod_{k=1}^d u_\delta(\cdot - (k-1)\sqrt\delta)$. Then for small $\delta$ we have $\deg h=d$ and $h(\theta)=1$ when $|\theta|\geq\ve$. By inspection of the Gagliardo semi-norm, it is easy to see that $|h|_{W^{1/p,p}}^p=d\, |u_\delta|_{W^{1/p,p}}^p+o(1)=C'_p\, d+o(1)$ as $\delta\to 0$.
% \end{proof}
\end{proof}

\begin{proof}[Proof of Lemma \ref{if1}]
We may assume that $d_1\ge 0$. Consider a maximal family $(B_j)_{1\le j\le J}$ of disjoint balls in $\sn$ of radius $1/(3n)$. For large $n$ we have $J\ge d_1$. Consider a smooth map $f_n:\sn\to\sn$ such that:
\ben
\item
$f_n=(1,0,\ldots, 0)$ outside $\cup B_j$.
\item
$\deg f_n=1$ on each $B_1,\ldots, B_{d_1}$.
\item
$\deg f_n=0$ and $f_n$ is onto on each $B_{d_1+1},\ldots, B_J$. 
\een

Then clearly $f_n$ has all the required properties.
\end{proof}

Finally, we present the
\begin{proof}[Proof of Lemma \ref{ih2}]
We work on a ball $B$ containing the origin, instead of $\sn$, and when the given point is the origin.  It suffices to establish the conclusion of the lemma when $f\in\wsp(B ; \R)$ is smooth in $\overline B$ and satisfies $f(0)=0$. 
By the Sobolev embeddings, we may assume that $1<p<\infty$ and $s=1+N/p$. 

Write $f=\sum_{j=1}^N x_j g_j$, with $g_j$ smooth. This is possible since $f(0)=0$. Then 
\be
\label{ma8}
\p_k\, [(1-H_\ve)\,  f-f]=-H_\ve\, \p_k f-\sum_{j=1}^N x_j \, \p_k H_\ve\, \, g_j\to 0\text{ in }W^{N/p,p}\text{ as }\delta\to 0;
\ee
this follows from properties \eqref{A.8} and \eqref{A.9} of $H_\ve$ and from the fact that the multiplication with a fixed smooth function is continuous in $W^{N/p,p}$.

Using \eqref{ma8}, we immediately obtain that $(1-H_\ve)\, f\to f$ in $W^{1+N/p, p}$ as $\ve\to 0$. On the other hand, $(1-H_\ve)\, f$ vanishes near the origin.
%Let $s$ and $p$ satisfy \eqref{ff1}. 
%\smallskip
%\noindent
%{\it Step 1.} Proof of the lemma when $s>1$ and $p>1$\\
%Let $t:=N/p$, so that $t\ge s-1$. Consider a sequence $(G^\delta)$ satisfying properties {\it 1}--{\it 5} in the proof of Lemma \ref{id1}, as well as the additional property {\it 6} in Lemma A.1, {\it with respect to the space $W^{t,p}$}.
%
%this follows from properties {\it 5} and {\it 6} of $G^\delta$ and from the fact that the multiplication with a fixed smooth function is continuous in $W^{t,p}$.
% 
% 
%
%\smallskip
%\noindent
%{\it Step 2.} Proof of the lemma when $s\le 1$\\
%Since the family $(f^\delta)$ constructed in Step 1 is uniformly bounded, Step 2 follows from Step 1 and Gagliardo-Nirenberg.
 \end{proof}

\end{document}